\documentclass[a4paper,11pt]{article}
 \usepackage[english]{babel}
 \usepackage{graphicx}
 \usepackage[latin1]{inputenc}
 \usepackage[T1]{fontenc}
 \usepackage{lmodern}
 \usepackage[normalem]{ulem}
 \usepackage{verbatim}
 \usepackage{bbm}
 \usepackage{ntheorem}
 \usepackage{stmaryrd}
 \usepackage{graphics}
 \usepackage{amsmath}
 \usepackage{amssymb}



\makeatletter
\newcommand\xleftrightarrow[2][]{%
  \ext@arrow 9999{\longleftrightarrowfill@}{#1}{#2}}
\newcommand\longleftrightarrowfill@{%
  \arrowfill@\leftarrow\relbar\rightarrow}
\makeatother

\theoremheaderfont{\scshape}
\theoremseparator{.}
\newtheorem{theorem}{Theorem}[section]
\newtheorem{thm}[theorem]{Theorem}
\newtheorem{lem}[theorem]{Lemma}

\newtheorem{conj}[theorem]{Conjecture}
\newtheorem{prop}[theorem]{Proposition}
\newtheorem{coro}[theorem]{Corollary}

\def\sqw{\hbox{\rlap{\leavevmode\raise.3ex\hbox{$\sqcap$}}$%
\sqcup$}}

\newcommand{\N}{\ensuremath{{{\mathbb N}}}}

\newcommand{\proptoop}{\ensuremath{\mathop{V_\infty^\pi}}}

\newcommand{\Z}{\ensuremath{\mathbb Z}}

\newcommand{\agit}{\ensuremath{\curvearrowright}}

\newcommand{\acts}{\ensuremath{\curvearrowright}}

\newcommand{\clusters}{\ensuremath{\mathfrak C}^\pi}

\newcommand{\R}{\ensuremath{\mathbb R}}

\newcommand{\Pp}{\ensuremath{\mathbb{P}}}

\newcommand{\G}{\ensuremath{\Gamma}}

\newcommand{\fleche}{\ensuremath{\longleftrightarrow}}

\newenvironment{rem}[1][Remark.]{\begin{trivlist}
\item[\hskip \labelsep {\itshape #1}]}{\end{trivlist}}

\newenvironment{comm}[1][Comment.]{\begin{trivlist}
\item[\hskip \labelsep {\itshape #1}]}{\end{trivlist}}

\newenvironment{exem}[1][Example.]{\begin{trivlist}
\item[\hskip \labelsep {\textit{#1}}]}{\end{trivlist}}

\newenvironment{bexem}[1][The Bernoulli example.]{\begin{trivlist}
\item[\hskip \labelsep {\textit{#1}}]}{\end{trivlist}}

\newenvironment{cex}[1][Counter-example.]{\begin{trivlist}
\item[\hskip \labelsep {\textit{#1}}]}{\end{trivlist}}

\newenvironment{defi}[1][Definition.]{\begin{trivlist}
\item[\hskip \labelsep {\textsc{#1}}]}{\end{trivlist}}

\newenvironment{nota}[1][Notation.]{\begin{trivlist}
\item[\hskip \labelsep {\textsc{#1}}]}{\end{trivlist}}

\newenvironment{proof}{  
    \vspace*{-.4em}  {\it Proof.}%
}{
    \hfill\sqw\vspace*{.5em}
}
\newcommand{\bde}{\begin{proof}\ }
\newcommand{\ede}{\end{proof}}

\date{\today}
\author{S\'ebastien \sc{Martineau}}
\title{Ergodicity and indistinguishability in percolation theory}
\begin{document}

\maketitle

\begin{abstract}
This paper explores the link between the ergodicity of the cluster equivalence relation restricted to its infinite locus and the indistinguishability of infinite clusters. It is an important element of the dictionary connecting  orbit equivalence and  percolation theory. This note starts with a short exposition of some standard material of these theories. Then, the classic correspondence between ergodicity and indistinguishability is presented. Finally, we introduce a notion of strong indistinguishability that corresponds to strong ergodicity, and obtain that this strong indistinguishability holds in the Bernoulli case. We also define an invariant percolation that is not insertion-tolerant, satisfies the Indistinguishability Property and does not satisfy the Strong Indistinguishability Property.
\end{abstract}

\section*{Introduction}

Orbit equivalence is a branch of ergodic theory that focuses on the dynamical properties of equivalence relations. Its fruitful interactions with other mathematical fields are numerous: operator algebra theory \cite{mvn, p}, foliation theory \cite{c,levitt}, descriptive set theory \cite{jkl,km}\dots\
Among the many concepts of the field, a  fundamental one is the notion of \emph{ergodicity}: an equivalence relation defined on a probability space is said to be ergodic if every saturated set has measure 0 or 1. It is striking to see how a definition that is usually given in the group action context can be easily stated in the seemingly static framework of equivalence relations.

The other fundamental notion considered in this note, \emph{indistinguishability}, belongs to percolation theory, a branch of statistical physics. Percolation is concerned with the study of random subgraphs of a given graph. These subgraphs are generally far from connected, and one is naturally interested in their infinite connected components --- or infinite clusters. A difficult theorem due to Lyons and Schramm \cite{ls} states that, under some hypotheses, if several infinite clusters are produced, they all ``look alike''. This is the Indistinguishability Theorem (see Theorem~\ref{ls}).

Its equivalence to some form of ergodicity should not be surprising: in both cases, when one asks a nice question, all the objects --- in one case the points of the space lying under the relation, in the other one the infinite clusters --- give the same answer. This connection is well-understood (see \cite{gdir,gl} and Proposition~\ref{ergind}).
In the orbit equivalence world, a hard theorem due to Chifan and Ioana (see \cite{ci} and Theorem~\ref{corocf}) allows to get from this ergodicity a \emph{stronger} form of ergodicity.

In this paper, we define a notion of strong indistinguishability and prove its equivalence to strong ergodicity: this is Theorem~\ref{ergindfort}. In particular, Bernoulli percolation satisfies the Strong Indistinguishability Property (see Corollary~\ref{corollaire}). We also define an invariant percolation that is not insertion-tolerant, satisfies the Indistinguishability Property and does not satisfy the Strong Indistinguishability Property (see Subsection~\ref{donotcoincide}). \begin{small}Indistinguishability results are usually hard to prove for non insertion-tolerant percolations: for instance, such a result is expected to hold for the Wired Uniform Spanning Forest but remains conjectural.\end{small}

\vspace{0.3 cm}

This note is self-contained, so that the orbit equivalence part can be read without prerequisite by a percolation theorist and vice versa. The first section presents what will be needed of orbit equivalence theory. The second one deals with percolation theory. The third and last section recalls the classic correspondence between ergodicity and indistinguishability and explores the correspondence between strong ergodicity and the notion of strong indistinguishability defined in this note.

\subsubsection*{Acknowledgements}

I would like to thank Vincent Beffara and Damien Gaboriau for the care with which they have followed and commented this work. I am also grateful to the orbit equivalence community of Lyon for fruitful discussions, in particular to Damien Gaboriau and Fran\c{c}ois Le Ma\^itre. I have been supported by the grants ANR MAC2 and ANR AGORA.

\subsubsection*{Terminology}

If $R$ is an equivalence relation defined on a set $X$, the \emph{$R$-class} of $x$ is
$$
[x]_R:=\{y\in X : xRy\}.
$$
A subset $A$ of $X$ is said to be \emph{$R$-saturated}, or \emph{$R$-invariant}, if 
$$
\forall x\in A, [x]_R\subset A.
$$
The \emph{$R$-saturation} of a subset $A$ of $X$ is the smallest subset $R$-saturated subset of $X$ that contains $A$. Concretely, it is $\bigcup_{x\in A} ~[x]_R$.

\newpage

\section{Orbit equivalence theory}

\setcounter{subsection}{-1}

This section presents standard definitions and theorems from orbit equivalence theory. For details relative to Subsection~\ref{count}, one can refer to \cite{k}. For subsections~\ref{cber} to \ref{sub:graphings}, possible references are \cite{gicm} and \cite{km}.

\subsection{Generalities on the standard Borel space}
\label{count}

A measurable space $X$ is called a \emph{standard Borel space} if it can be endowed with a Polish topology inducing its $\sigma$-algebra. For instance, $\{0,1\}^\N$ endowed with the product $\sigma$-algebra is a standard Borel space.
A measurable subset of a standard Borel space is called a \emph{Borel subset}.

\vspace{0.3 cm}

The following general results on standard Borel spaces will be used without explicit mention.

\begin{thm}
Any Borel subset of a standard Borel space is itself a standard Borel space.
\end{thm}
Let $X$ and $Y$ be two measurable spaces. A bijection $f:X\to Y$ is a \emph{Borel isomorphism} if $f$ and $f^{-1}$ are measurable. If $X=Y$, we speak of \emph{Borel automorphism}.

\begin{thm}
Let $X$ and $Y$ be standard Borel spaces. If $f:X\to Y$ is a measurable bijection, then $f^{-1}$ is automatically measurable, hence a Borel isomorphism.
\end{thm}

\begin{thm}
\label{unicbor}
Every non-countable standard Borel space is isomorphic to $[0,1]$. In particular, the continuum hypothesis holds for standard Borel spaces.
\end{thm}

\subsection{Countable Borel equivalence relations}
\label{cber}

Let $\G$ be a countable group and $\G\acts X$ be a Borel action of it on a standard Borel space. By \emph{Borel action}, we mean that every $\gamma\in\G$ induces a Borel automorphism of $X$.
Such an action induces a partition of $X$ into orbits. Let us consider $R$ (or $R_{\G\agit X}$) the relation ``being in the same orbit'' and call it the \emph{orbit equivalence relation} of $\G \acts X$. It is a subset of $X^2$. Since $\G$ is countable, the following assertions hold:
\begin{itemize}
\item $R$ is \emph{countable}, i.e.\ every $R$-class is (finite or) countable,

\item $R$ is Borel, as a subset of $X^2$.
\end{itemize}
The following theorem provides the converse:
\begin{thm}[Feldman-Moore, \cite{fm}]
\label{fm} 
Every countable Borel equivalence relation on a standard Borel space is induced by a Borel action of some countable group.
\end{thm}
In other words, every countable Borel equivalence relation on a standard Borel space is an orbit equivalence relation. This is why the theory of ``countable Borel equivalence relations'' is called ``orbit equivalence theory''.

\subsection{Measure invariance}

When dealing with a Borel action of $\G$ on a probability space, it makes sense to speak of invariance of the probability measure. The purpose of this subsection is to define this notion for countable Borel equivalence relations. To begin with, one needs to know how the standard Borel space behaves when it is endowed with a probability measure.

\begin{defi}
A \emph{standard probability space} is a standard Borel space endowed with a probability measure. 
\end{defi}

\begin{thm}
\label{unicprob}
Every atomless standard probability space $(X,\mu)$ is isomorphic to $[0,1]$ endowed with its Borel $\sigma$-algebra and the Lebesgue measure, i.e. there exists a measure-preserving Borel isomorphism between $(X,\mu)$ and $([0,1],\emph{d}x)$.
\end{thm}
\textsc{Throughout this paper, standard probability spaces will implicitly be assumed to be atomless.}

\vspace{0.5 cm}
Having a nice measured space to work on is not enough to provide a notion of invariance of the measure: to do so, one needs relevant transformations, presented below.

\begin{defi}
If $R$ is a countable Borel equivalence relation, $[R]$ denotes the group of the Borel automorphisms of $X$ whose graph is included in $R$.  A \emph{partial Borel automorphism} of $X$ is a Borel isomorphism between two Borel subsets of $X$. One denotes by $[[R]]$ the set of partial Borel automorphisms whose graph is included in $R$.
\end{defi}

\begin{rem}
In the literature, $X$ is often equipped with a ``nice'' probability measure\footnote{Here, ``nice'' means ``$R$-invariant'', which will be defined using $[R]$ (as defined above).}, and one often uses $[R]$ and $[[R]]$ to denote the objects defined above quotiented out by almost everywhere agreement. In this paper, we will stick to the definition we gave, which can be found in \cite{km}. 
\end{rem}
As exemplified by the theorem below, these Borel automorphisms allow us to mimic intrinsically the ``group action'' definitions in the ``orbit equivalence'' setting.

\begin{thm}
\label{preserv}
Let $R$ be a countable Borel equivalence relation on a standard probability space $(X,\mu)$. The following assertions are equivalent:

\begin{itemize}
\item there exist $\G$ a countable group and $\G\agit X$ a measure-preserving Borel action of it such that $R=R_{\G\agit X}$,

\item every  Borel action of a countable group that induces $R$ preserves $\mu$,

\item every element of $[R]$ preserves $\mu$.
\end{itemize}
When any of these equivalent properties is satisfied, we say that the measure $\mu$ is \emph{preserved} by $R$, or that it is $R$-\emph{invariant}.
\end{thm}

\vspace{0.5 cm}

\textsc{Henceforth,  $(X,\mu)$ will always be an atomless standard probability space and the equivalence relations we will consider on it will always be measure-preserving countable Borel equivalence relations.}

\vspace{0.5 cm}

\begin{rem}
There is no uniqueness theorem (analogous to Theorem~\ref{unicbor} or Theorem~\ref{unicprob}) for the object $(X,\mu,R)$. This is why orbit equivalence theory is not empty. Another fact to keep in mind  is that the space $X/R$ essentially never bears a natural standard Borel structure, even though $R$ is Borel.
\end{rem}

\subsection{Amenability and hyperfiniteness}
Amenability of a group can be defined in many equivalent ways. For our purpose, the following characterization will be enough.

\begin{thm}
A countable group $\G$ is amenable if and only if there exists a \emph{Reiter sequence}, i.e. $f_n\in \ell^1(\G)$ such that:
\begin{itemize}
\item $\forall n, f_n \geq 0$ and $\|f_n\|_1=1$,

\item $\forall \gamma\in\G, \|f_n-\gamma\cdot f_n\|_1\underset{n\to\infty}{\longrightarrow}0.$
\end{itemize}
\end{thm}
\begin{small}
In the theorem above, $\G$ acts on $\ell^1(\G)$ via $\gamma\cdot f(\eta):=f(\gamma^{-1}\eta)$. Taking the inverse of $\gamma$ guarantees that this defines a left action. Besides, the action it induces on indicator functions corresponds to the natural action $\G\acts\text{Subsets}(\G)$, i.e.\ we have $\gamma\cdot 1_A=1_{\gamma\cdot A}$.
\end{small}

\vspace{0.3 cm}

This theorem in mind, the following definition of amenability for equivalence relations is natural.

\begin{defi}
Let $R$ be a countable Borel equivalence relation on $(X,\mu)$. One says that $R$ is \emph{$\mu$-amenable} if and only if there exists a sequence of \emph{Borel} functions $f_n : R \to \mathbb{R}^+$ such that:
\begin{itemize}
\item $\forall x\in X,~ \sum_{y\in [x]_R} f_{n}(x,y)=1$,

\item there exists a full-measure $R$-invariant Borel subset $A\subset X$ such that $$\forall (x,y) \in (A\times A)\cap R, \sum_{z\in[x]_R}|f_n(x,z)-f_{n}(y,z)| \underset{n\to\infty}{\longrightarrow} 0.$$
\end{itemize}
\end{defi}

\begin{small}
\begin{comm}
In the definition above (and in others), one can indifferently impose $A$ to be $R$-invariant or not. Indeed, it can be deduced from Theorem~\ref{fm} that the $R$-saturation of a $\mu$-negligible set is still $\mu$-negligible. (Recall that all considered equivalence relations are tacitly assumed to preserve the measure.)
\end{comm}
\end{small}

Proposition~\ref{gener} shows that this definition is a nice extension of the classic notion of amenability (for countable groups) to equivalence relations.

\begin{nota}
Let $\G\acts X$ be  a Borel action of a countable group on a standard Borel space. If $X$ is endowed with an atomless probability measure $\mu$ that is $\G$-invariant, we will write $\G\acts (X,\mu)$. 
\end{nota}

\begin{prop}
\label{gener}
Let $\G\agit (X,\mu)$ be a measure-preserving action of a countable group. If $\G$ is amenable, then $R_{\G\agit X}$ is $\mu$-amenable. Besides, if $\G\acts X$ is free, then the converse holds. 
\end{prop}

It is easy to see that \emph{finite} equivalence relations (i.e.\ whose classes are finite) are amenable: one just needs to set $f_{n}(x,y)=\frac{1}{|[x]_R|}1_{y\in[x]_R}$. The proof  naturally extends to hyperfinite equivalence relations, defined below.

\begin{defi}
An equivalence relation $R$ on a standard Borel space $X$ is said to be \emph{hyperfinite} if it is a countable increasing union of finite Borel equivalence subrelations. (No measure appears in this definition.) If $\mu$ is an $R$-invariant probability measure on $X$, the relation $R$ is \emph{hyperfinite $\mu$-almost everywhere} if there exists a full-measure Borel subset $A\subset X$ such that $R\cap (A\times A)$ is hyperfinite. 
\end{defi}

\begin{exem}
The group $\G_\infty := \bigoplus_{n\in \N} \Z/2\Z$ is the increasing union of the subgroups $\G_N := \bigoplus_{n\leq N} \Z/2\Z$. Hence, any $R_{\G_\infty\agit X}$ is hyperfinite.  Besides, $\G_\infty$ is amenable: set $f_n = \frac{1_{\G_n}}{|\G_n|}$. Hence, any $R_{\G_\infty\agit (X,\mu)}$ is $\mu$-amenable.
\end{exem}

\begin{thm}[Connes-Feldman-Weiss, \cite{cfw}]
Let $R$ be a Borel countable equivalence relation on $(X,\mu)$. The relation $R$ is $\mu$-amenable if and only if it is hyperfinite $\mu$-almost everywhere.
%In other words, $\mu$-amenability and ``hyperfiniteness $\mu$-almost everywhere'' are the same notion.
\end{thm}

\subsection{Ergodicity}

\label{ergodicity}

\begin{defi}
Let $\G\acts(X,\mu)$ be a measure-preserving action. It is said to be \emph{ergodic} if, for every $\Gamma$-invariant Borel subset $B$ of $X$,  either $\mu(B)=0$ or $\mu(B)=1$.
\end{defi}

\begin{defi}
An equivalence relation $R$ on a standard probability space $(X,\mu)$ is said to be \emph{ergodic} (or \emph{$\mu$-ergodic}) if, for every $R$-invariant Borel subset $B$ of $X$,  either $\mu(B)=0$ or $\mu(B)=1$.
\end{defi}

\begin{rem}
Let $\G\acts (X,\mu)$ be a measure-preserving group action. Let $B$ be a subset of $X$. Notice that it is the same for $B$ to be $\G$-invariant or $R_{\G\acts X}$-invariant. This means that the following assertions are equivalent:
\begin{itemize}
\item $\forall \gamma\in\G, \gamma\cdot B= B$,

\item $\forall x\in B, \forall y \in X, x R_{\G\acts X}y \implies y\in B$.
\end{itemize}
In particular, $\G\acts X$ is ergodic if and only if $R_{\G\acts X}$ is ergodic.
\end{rem}

\begin{bexem}
Let $\G$ be an infinite countable group and $(\Sigma, \nu)$ denote either $([0,1],\text{Leb})$ or $(\{0,1\}, \text{Ber}(p))=(\{0,1\}, (1-p)\delta_0+p\delta_1)$. Let $A$ denote either $\G$ or the edge-set of a Cayley graph of $\G$. (The notion of Cayley graph is introduced in Subsection~\ref{gendef}.) Let $\mathcal{S}$ be the equivalence relation induced by the shift action of $\G$ on $ \left(\Sigma^A,\nu^{\otimes A}\right)$ defined by $$\gamma\cdot(\sigma_a)_{a\in A}=(\sigma_{\gamma^{-1}\cdot a})_{a\in A}.$$ This equivalence relation preserves $\nu^{\otimes A}$ and is ergodic.
\end{bexem}
The following theorem states that the amenable world shrinks to a point from the orbital point of view.

\begin{thm}[Dye, \cite{d}] Every countable Borel equivalence relation that is ergodic and hyperfinite $\mu$-almost everywhere is isomorphic to the orbit equivalence relation of the Bernoulli shift $\left(\Z\agit\left(\{0,1\}^\Z,\text{\emph{Ber}}(1/2)^{\otimes \Z}\right)\right)$. This means that if $R$ is such a relation on a standard probability space $(X,\mu)$, there exist
\begin{itemize}
\item a full-measure $R$-invariant Borel subset $A$ of $X$,

\item  a full-measure $\Z$-invariant Borel subset $B$ of $\{0,1\}^{\Z}$,

\item a measure-preserving Borel isomorphism $f:A\to B$
\end{itemize}
such that
$
\forall x,y \in A, x R y \Longleftrightarrow f(x)R_{\Z\agit \{0,1\}^{\Z}} f(y).
$
\end{thm}

\subsection{Strong ergodicity}

The notion of strong ergodicity, presented in this subsection, is due to Schmidt \cite{st}.

\begin{defi}
Let $\G\agit (X,\mu)$ be a measure-preserving action. A sequence $(B_n)$ of Borel subsets of $X$ is said to be \emph{asymptotically $\G$-invariant (with respect to $\mu$)} if
$$
\forall \gamma\in\G,~ \mu((\gamma\cdot B_n)\triangle B_n) \underset{n\to\infty}{\longrightarrow}0.
$$
The action $\G\agit (X,\mu)$ is said to be \emph{strongly ergodic} if, for every asymptotically $\G$-invariant sequence of Borel sets $(B_n)$, 
$$\mu(B_n)(1-\mu(B_n))\underset{n\to\infty}{\longrightarrow}0.$$
\end{defi}

Making use of $[R]$, one can extend this notion to equivalence relations.

\begin{defi}
Let $R$ be an equivalence relation on a standard probability space $(X,\mu)$.
A sequence $(B_n)$ of Borel subsets of $X$ is said to be \emph{asymptotically $R$-invariant (with respect to $\mu$)} if
$$
\forall \phi\in[R],~ \mu(\phi(B_n)\triangle B_n)\underset{n\to\infty}{\longrightarrow}0.
$$
The equivalence relation $R$ is said to be \emph{strongly ergodic} if, for every asymptotically $R$-invariant sequence of Borel sets $(B_n)$,
$$\mu(B_n)(1-\mu(B_n))\underset{n\to\infty}{\longrightarrow}0.$$
\end{defi}

\begin{rem}
One can check that if $\G\acts (X,\mu)$ is a measure-preserving action, then $(B_n)$ is asymptotically $\G$-invariant if and only if it is asymptotically $R_{\G\acts X}$-invariant. In particular, $\G\acts (X,\mu)$ is strongly ergodic if and only if $R_{\G\acts X}$ is strongly ergodic.
\end{rem}

\begin{rem}
It is clear that strong ergodicity implies ergodicity: if $B$ is invariant, set $B_n := B$ for all $n$ and apply strong ergodicity. What may be less clear is that the converse does not hold. In fact, the unique ergodic amenable relation is not strongly ergodic. To prove this, consider an ergodic measure-preserving action of $\Gamma_\infty := \bigoplus_{n\in \N} \Z/2\Z$ on a standard probability space $(X,\mu)$, for example the Bernoulli shift. For $N\in \N$, set as previously $\G_N := \bigoplus_{n\leq N} \Z/2\Z$. Since $\G_N$ is finite, the restricted action $\G_N\acts (X,\mu)$ admits a fundamental domain $D_N$, that is a Borel subset that intersects each orbit in exactly one point\footnote{To get such a fundamental domain, one can think of  $X$ as  $[0,1]$ and keep a point iff it is the smallest in its orbit for the usual ordering of the interval.}. One can find a Borel subset of $D_N$ of measure $\frac{\mu(D_N)}2$. Then, define $B_N$ as the $R_{\G_N\acts X}$-saturation of $D_N$. Each $B_N$ has measure $\frac 1 2$ and  is $\Gamma_M$-invariant for $M\leq N$, which completes the demonstration.
\end{rem}

%\begin{thm}[Chifan-Ioana]
%Let $R$ be a subrelation of $\mathcal{S}$. Then, there exists a partition $(X_n)_{n\in{\mathbb Z}_+}$ of $\Omega^\G$ such that
%\begin{itemize}
%\item for all $n\geq0$, $X_n$ is an $R$-invariant Borel subset of $\Omega^\G$,

%\item $R_{|X_0}$ is hyperfinite,

%\item for all $n\geq 1$, either $X_n$ has measure 0 or $R_{X_n}$ is $\frac{\mu}{\mu(X_n)}$-strongly ergodic (thus ergodic).
%\end{itemize}
%\end{thm}
%We will only use a very small part of this result, explicited below.

\vspace{0.2 cm}

The following theorem will be crucial in Section~\ref{third} because it allows, under certain conditions, to deduce strong ergodicity from ergodicity. In its statement, $\mathcal S$ stands for the relation introduced in the Bernoulli example of Subsection~\ref{ergodicity} and $(X,\mu)$ for its underlying standard probability space.

\begin{thm}[Chifan-Ioana, \cite{ci}]
\label{corocf}
Let $B$ be a non-$\mu$-negligible Borel subset of $X$. Any ergodic equivalence subrelation of $\left(\mathcal{S}_{|B},\frac{\mu}{\mu(B)}\right)$ that is not $\frac{\mu}{\mu(B)}$-amenable is strongly ergodic.
\end{thm}

\begin{comm}
In fact, \cite{ci} proves a lot more. But since we do not need the full result of Chifan and Ioana --- whose statement is more technical ---, we will stick to the stated version.
\end{comm}

\subsection{Graphings}
\label{sub:graphings}

A \emph{graphing} of a relation $R$ on $X$ is a countable family $(\varphi_i)$ of partial Borel automorphisms of $X$ that generates $R$ as an equivalence relation: this means that the smallest equivalence relation on $X$ that contains the graphs of the $\varphi_i$'s is $R$. In particular, the Borel partial automorphisms that appear in a graphing belong to $[[R]]$.
The notion of graphing generalizes to relations the notion of generating system.

Notice that the data of a graphing endows each $R$-class with a structure of connected graph: put an edge from $x$ to $x'$ if there is an $i$ such that $x$ belongs to the domain of $\varphi_i$ and $x'=\varphi_i(x)$. One can do this with multiplicity.

\begin{exem}
Let $\G$ be a finitely generated group and $S$ a finite generating system of $\Gamma$. Let $\G\agit X$ be a Borel action on a standard Borel space. For $s\in S$, let $\varphi_s$ denote the Borel automorphism implementing the action of $s^{-1}$. Then, $(\varphi_s)_{s\in S}$ is a graphing of $R_{\G\agit X}$.  Let us take a closer look at the graph structure.

Let $\mathcal G=(V,E)=(\G,E)$ denote the Cayley graph of $\G$ relative to $S$ (see Subsection~\ref{gendef} for the definition). In this example, we will use the concrete definition of Cayley graphs and take the vertex-set to be $\G$. If the action is free, then, for every $x$, the mapping $\gamma\mapsto \gamma^{-1}\cdot x$ is a graph isomorphism between $\mathcal G$ and the graphed orbit of $x$. The only point to check is that the graph structure is preserved: for all $(\gamma,\eta,x)\in \G\times\G\times X$,
\begin{eqnarray*}
(\gamma,\eta ) \in E& \Longleftrightarrow & \exists s\in S, \eta=\gamma s\\
 & \Longleftrightarrow & \exists s\in S, \eta^{-1}=s^{-1}\gamma^{-1}\\
 & \Longleftrightarrow & \exists s\in S, \eta^{-1}\cdot x=s^{-1}\gamma^{-1}\cdot x\\
 & \Longleftrightarrow & (\eta^{-1}\cdot x,\gamma^{-1}\cdot x)\text{ is an edge}.\\
\end{eqnarray*}
The point in putting all these inverses is that, this way, we only work with Cayley graphs on which the group acts from the left%\footnote{The generators of $\G$ act on $X$ from the left while two vertices in $\G$ are declared adjacent if they ``differ from the right'' by a generator. The mapping $\gamma\mapsto \gamma^{-1}$, reversing left and right, solves the problem.}
. If the action is not assumed to be free, the map $\gamma\mapsto \gamma^{-1}\cdot x$ is only a graph-covering.
\end{exem}

To describe how a graph behaves at infinity, a useful notion is the one of end.

\begin{defi}
Let $\mathcal{G}=(V,E)$ be a countable graph. An \emph{end} of $\mathcal{G}$ is a map $\xi$ that associates to each finite subset $K$ of $V$ an infinite connected component of its complement, and that satisfies the following compatibility condition:
$$
\forall K,K',~K\subset K' \implies \xi(K')\subset \xi(K).
$$
\end{defi}

\begin{rem}
Every end is realized by some infinite injective path: for every $\xi$, there exists an infinite injective path $c:\N\to V$ such that, for every finite subset $K$ of $V$, the path $c$ eventually lies in $\xi(K)$. This results from a diagonal extraction argument.
\end{rem}

We now have all the vocabulary needed to state the following theorem, the graph-theoretic flavor of which will allow us to travel between the world of orbit equivalence and the one of percolation.

\begin{thm}
\label{gbouts}
Let $R$ be a countable Borel equivalence relation on $X$ that preserves the atomless probability measure $\mu$.
\begin{itemize}
\item If it admits a graphing such that, for $\mu$-almost every $x$, the class of $x$ has two ends (seen as a graph), then $R$ is hyperfinite $\mu$-almost everywhere.

\item If it admits a graphing such that, for $\mu$-almost every $x$, the class of $x$ has infinitely many ends, then $R$ is not ``hyperfinite $\mu$-almost everywhere''.
\end{itemize}
\end{thm}

This theorem is corollaire IV.24 in \cite{gcost}. It is a statement among several of the kind (see \cite{a,ghys}).

\newpage

\section{Percolation}

Percolation is a topic coming originally from statistical mechanics (see \cite{grim}). After a foundational paper by Benjamini and Schramm \cite{bs}, strong connections with group theory have developed. This section presents the objects and theorems that will be needed in Section~\ref{third}. For more information about this material, one can refer to \cite{gdir}, \cite{l} and \cite{lp}.

\subsection{General definitions}
\label{gendef}

\vspace{0.5 cm}

\textsc{From here on, $\G$ will be assumed to be finitely generated.}

\vspace{0.5 cm}

 Let $S$ be a finite generating set of $\G$. Define a graph by taking $\G$ as vertex-set and putting, for each $\gamma\in\G$ and $s\in S$, an edge from $\gamma$ to $\gamma s$.  This defines a locally finite connected graph $\mathcal G = (V,E)$ that is called the \emph{Cayley graph} of $\G$ relative to $S$. The action of $\G$ on itself by multiplication from the left induces a (left) action on $\mathcal G$ by graph automorphisms. It is free and transitive as an action on the vertex-set. In fact, a locally finite connected graph $\mathcal G$ is a Cayley graph of $\G$ if and only if $\G$ admits an action on $\mathcal G$ that is free and transitive on the vertex-set.

We have defined $\mathcal G$ explicitly to prove that $\G$ admits Cayley graphs, but further reasonings shall be clearer if one forgets that $V=\G$ and just remembers that $\mathcal G$ is endowed with a free vertex-transitive action of $\G$. Thus, in order to get an element of $\G$ from a vertex, one will need a reference point.
Let $\rho$ be a vertex of $\mathcal G$ that we shall use as such a reference or anchor point.  Any vertex $v\in V$ can be written uniquely in the form $\gamma\cdot\rho$. 

The action $\G\agit E$ induces a shift action $\G\acts \Omega:=\{0,1\}^E$.
A \emph{(bond) percolation} will be a probability measure on $\Omega$. It is said  to be \emph{$\G$-invariant} if it is as a probability measure on $\Omega$.

\vspace{0.5 cm}

\textsc{In what follows, all considered percolations will be assumed to be $\G$-invariant.
Besides, for simplicity, we will work under the implicit assumption that $\Pp$ is atomless, so that $(\Omega,\Pp)$ will always be a standard probability space. }
%This assumption will not be satisfied by the Dirac masses $\Pp_0$ and $\Pp_1$ of subsection~\ref{percoindep}, but this will not raise any issue since the corresponding subsection does not rely on orbit equivalence.

\vspace{0.5 cm}

A point $\omega$ of $\Omega$ is seen as a subgraph of $\mathcal G$ the following way: $V$ is its set of vertices and $\omega^{-1}(\{1\})$ its set of edges. In words, keep all edges whose label is 1 and throw away the others --- edges labeled 1 are said to be \emph{open}, the other ones are said to be \emph{closed}. The connected components of this graph are called the \emph{clusters} of $\omega$. 
If $v\in V$, its $\omega$-cluster will be denoted by $\mathcal{C}(\omega,v)$. For $v\in V$, the map $\omega\mapsto \mathcal{C}(\omega,v)$ is Borel, the set of finite paths in $\mathcal G$ being countable. If $(u,v)\in V^2$, we will use $u\underset{\omega}{\fleche}v$ as an abbreviation for  ``$u$ and $v$ are in the same $\omega$-cluster''. The number of infinite clusters of $\omega$ will be denoted by $N_\infty(\omega)$.  The function $N_\infty$ is Borel.

\subsection{Independent percolation}
\label{percoindep}

The simplest interesting example of percolation is the product measure $\text{Ber}(p)^{\otimes E}$, for $p\in(0,1)$. It will be denoted by $\Pp_p$. Such percolations are called \emph{independant} or \emph{Bernoulli} percolations.
%\begin{rem}Notice that $\Pp_0$ and $\Pp_1$ are Dirac masses, hence break the atomlessness implicit assumption of the previous subsection. Though, this assumption was only made to allow us to use orbit equivalence theory, since a corresponding assumption was made in this setting. Since the current subsection does not rely on orbit equivalence theory and $\Pp_0$ and $\Pp_1$ will not be encountered in further parts, no vicious circle appears.\end{rem}
One is interested in the emergence of infinite clusters when $p$ increases. To study this phenomenon, introduce the \emph{percolation function} of $\mathcal{G}$, defined as $$\theta_\mathcal{G}:p\mapsto \Pp_p[|\mathcal{C}(\omega,\rho)|=\infty].$$

Endow $[0,1]^E$ with the probability measure $\Pp_{[0,1]} :=\text{Leb}([0,1])^{\otimes E}$. Notice that $\Pp_p$ is the push-forward of $\Pp_{[0,1]}$ by the following map
$$
\begin{array}{lccl}
\pi_p : & [0,1]^E & \longrightarrow & \{0,1\}^E \\
    & x & \longmapsto & (1_{x(e) < p})_{e\in E}.\end{array}
$$
Realizing probability measures as distributions of random variables suitably defined on a same probability space is called a \emph{coupling}. A fundamental property of this coupling is that, when $x\in[0,1]^E$ is fixed, $p\mapsto \pi_p(x)$ is non-decreasing for the product order. One deduces the following proposition.

\begin{prop}
The function $\theta_\mathcal{G}$ is non-decreasing.
\end{prop}

\begin{coro}
There exists a unique real number $p_c(\mathcal{G})\in[0,1]$ such that the following two conditions hold:
\begin{itemize}
\item $\forall p < p_c(\mathcal{G}),~ \theta_\mathcal{G}(p)=0$,

\item $\forall p > p_c(\mathcal{G}),~ \theta_\mathcal{G}(p)>0$.
\end{itemize}
One calls $p_c(\mathcal{G})$  the \emph{critical probability} of $\mathcal{G}$.
\end{coro}

\begin{rem}
When $p_c(\mathcal{G})$ is not trivial (neither  0 nor 1), this result establishes the existence of a \emph{phase transition}. One cannot have $p_c(\mathcal{G}) = 0$, but $p_c(\mathcal{G})=1$ may occur (e.g.\ it does for $\mathbb{Z}$).
\end{rem}

The following theorems describe almost totally the phase transitions related to the number of infinite clusters.

\begin{prop}
For all $p\in(0,1)$, the random variable $N_\infty$ takes a $\Pp_p$-almost deterministic value, which is 0, 1 or $\infty$. This value is 0 if $p< p_c(\mathcal{G})$ and 1 or $\infty$ if $p>p_c(\mathcal{G})$.
\end{prop}

%\begin{rem}The proof of this proposition uses the ergodicity and insertion-tolerance\footnote{This property is defined in subsection~\ref{indisclus}.} of $\Pp_p$.\end{rem}

\begin{thm}[H\"aggstr\"om-Peres, \cite{hp}]
There exists a unique real number $p_u(\mathcal{G})\in[p_c(\mathcal{G}),1]$ such that the following two conditions hold:
\begin{itemize}
\item $\forall p < p_u(\mathcal{G}), \Pp_{p}[N_\infty=1]=0$,

\item $\forall p > p_u(\mathcal{G}), \Pp_{p}[N_\infty=1]=1$.
\end{itemize}
One calls $p_u(\mathcal{G})$  the \emph{uniqueness probability} of $\mathcal{G}$.
\end{thm}

\begin{rem}
If $\G$ is amenable, Proposition~\ref{unicite} gives $p_c(\mathcal{G})=p_u(\mathcal{G})$. The converse is conjectured to hold. A weak form of the converse has been established by Pak and Smirnova-Nagnibeda \cite{psn} and used in \cite{gl}.
\end{rem}

\begin{prop}[\cite{blps2}]
If $\G$ is non-amenable, then $p_c(\mathcal{G})<1$ and there is no infinite cluster $\Pp_{p_c(\mathcal{G})}$-almost surely.
\end{prop}

\begin{conj}
If $p_c(\mathcal{G})<1$, then there is no infinite cluster $\Pp_{p_c(\mathcal{G})}$-almost surely.
\end{conj}

\vspace{0.25 cm}
The phase transition theorems are roughly summarized in the picture below. Remember that the quantities $p_c$, $p_u$ and 1 may coincide.

\begin{center}
\setlength{\unitlength}{1mm}
\thicklines
\begin{picture}(100,10)
\put(0,0){
   \line(1,0){100}
}

\put(33,5){$p_c$}
\put(33,-1.5){
   \line(0,1){3}
}

\put(67,5){$p_u$}
\put(67,-1.5){
   \line(0,1){3}
}

\put(0,5){$0$}
\put(0,-1.5){
   \line(0,1){3}
}

\put(100,5){1}
\put(100,-1.5){
   \line(0,1){3}
}

\put(10,3){$N_\infty = 0$}

\put(43,3){$N_\infty = \infty$}

\put(77,3){$N_\infty = 1$}
\end{picture}
\end{center}

\vspace{0.1 cm}

\subsection{Generalized percolation}

The notion of generalized percolation that is presented in this subsection is due to Gaboriau \cite{gdir}.

\vspace{0.3 cm}

Let $\G \agit (X,\Pp)$ be a Borel action on a standard probability space. Assume that it is provided together with a $\G$-equivariant map $$\pi : X \to \Omega = \{0,1\}^E,$$
the space $\{0,1\}^E$ being endowed with the shift action. This will be called a \emph{generalized ($\G$-invariant) percolation}. As for percolations, we will omit the ``$\G$-invariant'' part of the denomination.

To begin with, let us see how this notion is connected to the one presented in Subsection~\ref{gendef}. If a generalized percolation is given, then $\pi_\star \Pp$ --- the pushforward of $\Pp$ by $\pi$ --- is a $\G$-invariant percolation that may have atoms. Conversely, if one is given a $\G$-invariant atomless  percolation, one can consider the Bernoulli shift action $\G \agit X = \Omega$ together with $\pi : X \to \Omega$ the identity. Via this procedure, one can redefine in the percolation setting any notion introduced in the generalized framework.

Notice that the $\pi_p$'s of the standard coupling, introduced at the beginning of Subsection~\ref{percoindep}, provide interesting examples of such generalized percolations.

This setting provides the same atomless measures on $\Omega$ as the previous one, but it allows more flexibility in our way to speak of them. In the next subsection, we will discuss properties of clusters. The usual setting allows to speak of properties such as ``being infinite'', ``having three ends'', ``being transient for simple random walk''. The generalized one will allow us, if we consider $\G\agit [0,1]^E$ together with $\pi_{p_1}$, to speak of ``the considered $p_1$-cluster contains an infinite $p_0$-cluster''.

\subsection{Cluster indistinguishability}
\label{indisclus}

In this subsection, we work with a given generalized percolation. The action is denoted by $\G\agit (X,\Pp)$ and the equivariant map by $\pi$. 

\begin{nota}
We call \emph{vertex property} --- or \emph{property} ---  a Borel \textsc{$\G$-invariant} Boolean function on $X\times V$, i.e. a Borel function
$$
P:X\times V \to \{\text{true},\text{false}\}
$$
that is invariant under the diagonal action of $\G$. If $W\subset V$,   we write $P^+(x,W)$ for ``all the vertices in $W$ satisfy $P(x,.)$''. More formally, we define
 $$P^+(x,W) := ``\forall v\in W, P(x,v)\text{''}.$$ We also set
\begin{itemize}
\item $P^-(x,W) := ``\forall v\in W, \lnot P(x,v)$'',

\item $P^\pm(x,W) := ``P^+(x,W)\vee P^-(x,W)$''.
\end{itemize}
The expression $P^\pm(x,W)$ means ``all the vertices in $W$ agree on $P(x,.)$''.
\end{nota}

\begin{exem}The degree of a vertex in a graph is its number of neighbors.
 ``The vertex $v$ has degree 4 in $\pi(x)$ seen as a subgraph of $\mathcal G$'' is a property.
\end{exem}

\begin{defi}
We call \emph{cluster property} a property $P$ such that $P(x,v) \Longleftrightarrow P(x,u)$ as soon as $u\overset{\pi(x)}{\fleche}v$. In words, it is a vertex property such that, for any $x$, the function $P(x,.)$ is constant on $\pi(x)$-clusters.
\end{defi}

\begin{exem}
The previous example is usually not a cluster property: for most Cayley graphs $\mathcal{G}$, there exist subgraphs of $\mathcal G$ where some component has some vertices of degree 4 and others of other degree.  ``The $\pi(x)$-cluster of $v$ is infinite'', ``the $\pi(x)$-cluster of $v$ is transient'', ``the $\pi(x)$-cluster of $v$ has a vertex of degree 4'' are cluster properties.
\end{exem}

\begin{cex}
``The $\pi(x)$-cluster of $v$ contains $\rho$'' is \emph{not} a cluster property, because of the lack of $\G$-invariance. It is to avoid such ``properties'' that $\G$-invariance is required in the definition of vertex properties: allowing them would automatically make any indistinguishability theorem false since they can distinguish the cluster of the origin from the others.
\end{cex}

\begin{exem}
Here is another example of cluster property, which can be (directly) considered only in the generalized setting. Consider $X=[0,1]^E$ and $0<p_0<p_1<1$. We take $\pi=\pi_{p_1}$ (see Subsection~\ref{percoindep}). The property ``the $\pi_{p_1}(x)$-cluster of $v$ contains an infinite $\pi_{p_0}(x)$-cluster'' is a cluster property. It has been considered by Lyons and Schramm in \cite{ls} to derive the H\"aggstr\"om-Peres Theorem from indistinguishability.
\end{exem}

To formalize the indistinguishability of infinite clusters, one needs to speak of cluster properties and infinite clusters. Thus, we set  $$\proptoop(x):=  \{v\in V :|\mathcal{C}(\pi(x),v)|=\infty\}.$$

\begin{defi}
The considered generalized percolation will be said to satisfy \emph{(infinite cluster) indistinguishability} (or one will say that its infinite clusters are indistinguishable) if, for every cluster property $P$,
$$
\Pp[P^\pm(x,\proptoop(x))]=1.
$$
\end{defi}
Of course, this notion is empty as soon as $\Pp[N_\infty(\pi(x))\leq 1]=1$, e.g.\ for $\Pp_p$ when $\G$ is amenable.

\begin{rem}
Assume momentarily that $\G\agit (X,\Pp)$ is ergodic 
and that the infinite clusters are indistinguishable. Then for every cluster property $P$, by indistinguishability, 
$$
\Pp[P^+(x,\proptoop(x)) \text{ or } P^-(x,\proptoop(x))]=1.
$$
Besides, by ergodicity,
$\Pp[P^+(x,\proptoop(x))]$ and $ \Pp[P^-(x,\proptoop(x))]$ are 0 or 1. Altogether, these identities guarantee that
$$
\Pp[P^+(x,\proptoop(x))]=1\text{~~~or~~~}\Pp[P^-(x,\proptoop(x))]=1.
$$
%\begin{small}In rough words, indistinguishability allows an interversion ``for all $v$/excluded middle'' (\footnote{$\Pp[\forall v\in \proptoop(x), P(x,v)\text{ or }\lnot P(x,v)]=1$ becomes $$\Pp[\forall v\in \proptoop(x), P(x,v)\text{ or }\forall v\in \proptoop(x)\lnot P(x,v)]=1$$}) and the ergodicity of the \emph{action} makes legit an interversion ``excluded middle/probability 1''.\end{small}
\end{rem}

To state the Indistinguishability Theorem in its natural form, we need to introduce the notion of insertion-tolerance.

\subsection{Insertion-tolerance}

In this subsection, we work with non-generalized percolations.

\begin{defi}
If $(\omega,e)\in \Omega\times E$, one denotes by $\omega^e$ the unique element of $\Omega$ equal to $\omega$ on $E\backslash \{e\}$ and taking the value 1 at $e$. One sets $\Pi^e : \omega \mapsto \omega^e$.
A percolation is said to be \emph{insertion-tolerant} if for every Borel subset $B\subset \Omega$, for every edge $e$, $$\Pp[B]>0\implies\Pp[\Pi^e(B)]>0.$$
\end{defi}

\begin{exem}
For any $p\in(0,1)$, the percolation $\Pp_p$ is insertion-tolerant.
\end{exem}

\begin{prop}
\label{unicite}
If $\G$ is amenable and if $\Pp$ is an insertion-tolerant percolation on $\mathcal G$, then $\Pp[N_\infty(\omega)\leq 1]=1$.
\end{prop}

\begin{rem}
Proposition~\ref{unicite} improves results obtained in \cite{bk, gkn}. For a proof of the general statement, see \cite{lp}.
\end{rem}

\begin{prop}[\cite{ls}, Proposition 3.10]
\label{lsbouts}
If $\Pp$ is an insertion-tolerant percolation on $\mathcal G$ that produces a.s. at least two infinite clusters, then it produces a.s. infinitely many infinite clusters and each of them has infinitely many ends.
\end{prop}

Now that insertion-tolerance has been introduced, we can state the Indistinguishability Theorem of Lyons and Schramm (\cite{ls}).

\begin{thm}[Lyons-Schramm,~\cite{ls}]
\label{ls}
Any insertion-tolerant percolation has indistinguishable infinite clusters.
\end{thm}

%\begin{rem}In fact, this theorem holds for generalized percolations of some particular type, that moreover satisfy a form of insertion-tolerance. The precise statement can be found in \cite{ls}, remark 3.4.\end{rem}

\subsection{Percolation and orbit equivalence}
\label{percorbit}

In this subsection, we work with a generalized percolation, where the action is denoted by $\G\agit (X,\Pp)$ and the equivariant map by $\pi$. 

The \emph{cluster equivalence relation} is defined as follows: two configurations $x$ and $x'$ in $X$ are said to be $R_{cl}$\emph{-equivalent} if there exists $\gamma\in\G$ such that $\gamma^{-1}\cdot x=x'$ and $\gamma\cdot \rho \overset{\pi(x)}{\fleche}\rho$.
%In other words, two configurations are $R_{cl}$-equivalent if they are the same up to a translation moving back to $\rho$ a vertex of its cluster.
In words, an $R_{cl}$-class is a configuration up to $\G$-translation and with a distinguished cluster --- the one of the root $\rho$.

Every generalized percolation is $R_{cl}$-invariant, since $R_{cl}$ is a subrelation of $R_{\G\agit X}$.

Let $S$ denote the generating set associated to the choice of the Cayley graph $\mathcal G$. For $s \in S$, let $\tilde \varphi_s$ denote the restriction of $ x \mapsto s^{-1}\cdot x$ to the $x$'s such that the edge $(\rho,s\cdot\rho)$ is $\pi(x)$-open. If the action of $\Gamma$ on $X$ is free, this graphing induces on $[x]_{R_{cl}}$ the graph structure of the $\pi(x)$-cluster of the anchor point $\rho$. This remark, together with Theorem~\ref{gbouts} and Proposition~\ref{lsbouts}, provides the following proposition.

\begin{prop}
\label{themm}
Let $\Pp$ denote an insertion-tolerant classic percolation. Assume that
\begin{itemize}
\item $N_\infty$ is infinite $\Pp$-almost surely,
\item for $\Pp$-almost every $\omega$, the map $\gamma\mapsto \gamma\cdot \omega$ is injective.
\end{itemize}
Then $R_{cl}$ is not $\Pp$-amenable.
\end{prop}
\begin{rem}
This proposition applies to Bernoulli percolations that yield infinitely many infinite clusters.
\end{rem}

\newpage

\section{Ergodicity and indistinguishability}
\label{third}

Throughout this section, we will work with a generalized percolation. The underlying standard probability space will be denoted by $(X,\Pp)$ and the equivariant map by $\pi$.

\setcounter{subsection}{-1}

\subsection{The dictionary}
\label{dico}

The following array presents concisely the correspondence between percolation theory and orbit equivalence theory. In the following sections, no knowledge of this array will be assumed and we will start from scratch. Though, we think it may be useful to the reader to have all the data compactly presented in a single place, hence this subsection.

In the following ``dictionary'', the bijection $\psi:\Gamma\backslash (X\times V) \to X$ induced by $(x,\gamma\cdot \rho)\mapsto \gamma^{-1}\cdot x$ is the translator.

\vspace{0.7 cm}

\begin{center}
\begin{tabular}{|ccc|}
  \hline
  Orbit equivalence &  & Percolation\\
  \hline

$X$ &$\overset{\psi}{\fleche} $     &  $\Gamma\backslash(X\times V)$         \\

$\gamma^{-1}\cdot x$ & & $[(x,\gamma\cdot \rho)]$\\

$x\in X_\infty$&&$\rho\overset{\pi(x)}{\fleche}\infty$\\

Borel subset & & vertex property\\

$R_{cl}$-class & & cluster \\

$R_{cl}$-invariant & & cluster property \\

 ergodicity of $R$ &$\simeq$ & indistinguishability \\

$\phi$ s.t. $\text{gr}(\phi)\subset R_{cl}$ & & rerooting \\

$\phi \in [R]$ && vertex-bijective rerooting\\

asymptotically $R_{cl}$-invariant  && asymptotic cluster property \\

strong ergodicity of $R$ &$\simeq$& strong indistinguishability \\

graphing & & graph structure\\

  \hline
\end{tabular}
\end{center}

\vspace{0.2 cm}

\subsection{Classic connection}

The map $P\mapsto B_P:=\{x\in X : P(x,\rho)\}$ realizes a bijection from the set of properties onto the set of Borel subsets of $X$. Its inverse map is $B\mapsto \left(P_B:(x,\gamma\cdot \rho)\mapsto ``(\gamma^{-1}\cdot x,\rho)\in B\text{''}\right)$. It induces a bijection between the set of cluster properties and the set of $R_{cl}$-invariant Borel subsets of $X$.

\begin{lem}
\label{lemmaF}
Let $P$ denote a property and $\Lambda$ a subset of $\Gamma$. For any $x\in X$,
$$
P^\pm(x,\proptoop(x)\cap (\Lambda^{-1}\cdot\rho)) \Longleftrightarrow \forall y,z\in X_\infty \cap (\Lambda\cdot x),~ (y\in B_P\Longleftrightarrow z\in B_P).
$$
\end{lem}

\vspace{0.2 cm}

\begin{proof}
It results from the fact that, for any cluster property $P$ and any $x\in X$, if one sets $\Delta:= \Lambda^{-1}$,
\begin{equation*}
\begin{array}{l}
P^\pm(x,\proptoop(x)\cap (\Delta\cdot\rho))\Longleftrightarrow\left\{\forall u,v\in\proptoop(x)\cap (\Delta\cdot\rho),P(x,u)\Longleftrightarrow P(x,v)\right\}\\
~~\Longleftrightarrow~\left(\forall \gamma_0,\gamma_1\in \Delta,~\left\{\begin{tiny}\begin{array}{c} \gamma_0\cdot\rho\overset{\pi(x)}{\longleftrightarrow}\infty \\ \text{ and }\\ \gamma_1\cdot\rho\overset{\pi(x)}{\longleftrightarrow}\infty\end{array}\end{tiny}\right\} \implies \left(P(x,\gamma_0\cdot\rho)\Longleftrightarrow (P(x,\gamma_1\cdot\rho)\right) \right)\\
~~\Leftrightarrow~\forall \gamma_0,\gamma_1\in \Delta,~\left\{\begin{tiny}\begin{array}{c}\rho\xleftrightarrow{\pi(\gamma_0^{-1}\cdot x)}\infty \\ \text{ and }\\ \rho\xleftrightarrow{\pi(\gamma_1^{-1}\cdot x)}\infty\end{array}\end{tiny}\right\} \implies \left(P(\gamma_0^{-1}\cdot x,\rho)\Longleftrightarrow (P(\gamma_1^{-1}\cdot x,\rho)\right) \\
~~\Longleftrightarrow \forall y,z\in X_\infty \cap (\Lambda\cdot x),~ (y\in B_P\Longleftrightarrow z\in B_P).\hfill\phantom{word}\\
\end{array}
\end{equation*}
\end{proof}

Taking $\Lambda=\Gamma$ gives the following proposition.

\begin{prop}
\label{one}
Consider a generalized percolation defined by $\G\agit (X,\Pp)$  and a $\G$-equivariant map $\pi : X\to\Omega$.
Then the considered generalized percolation has indistinguishable infinite clusters if and only if for every Borel subset $B$ of $X$, for $\Pp$-almost every $x\in X$, the following holds: $$\forall y\in X_\infty\cap (\Gamma\cdot x), ~x\in B \Longleftrightarrow y\in B.$$
\end{prop}

We define the \emph{infinite locus} as $$X_{\infty}:= \{x\in X : |\mathcal{C}(\pi(x),\rho)|=\infty\}.$$
This definition coincides with the usual orbit-equivalence definition $$\{x\in X : |[x]_{R_{cl}}|=\infty\}$$ as soon as $\Gamma\acts X$ is free. 
Remember that if there is no $\pi$ in the second description, it is because it is hidden in $R_{cl}$.
Let $R$ denote the restriction of $R_{cl}$ to $X_\infty\times X_\infty$.

\begin{prop}
\label{two}
Consider a generalized percolation defined by $\G\agit (X,\Pp)$  and a $\G$-equivariant map $\pi : X\to\Omega$.
Assume that  $\Pp[X_{\infty}]>0$. Then $R$ is $\frac{\Pp}{\Pp[X_\infty]}$-ergodic if and only if for every cluster property $P$, the conditional probability $\Pp\left[P(x,\rho)|\rho\overset{\pi(x)}{\longleftrightarrow}\infty\right]$ is either 0 or 1.
\end{prop}
\begin{proof}
The relation $R$ is $\frac{\Pp}{\Pp[X_\infty]}$-ergodic if and only if, for every $R_{cl}$-invariant Borel subset $B$ of $X$, $\Pp[B\cap X_\infty]\in\{0,\Pp[X_\infty]\}$. The proposition results from the fact that, for any $R_{cl}$-invariant Borel subset of $X$ and any $x\in X$,
$$
\Pp[B\cap X_\infty]\in\{0,\Pp[X_\infty]\}\Leftrightarrow \Pp[P_B(x,\rho)\text{ and }\rho\overset{\pi(x)}{\longleftrightarrow}\infty]\in\left\{0,\Pp\left[\rho\overset{\pi(x)}{\longleftrightarrow}\infty\right]\right\}.
$$
\end{proof}

\begin{prop}[Gaboriau-Lyons, \cite{gl}]
\label{ergind}
Consider a generalized percolation defined by $\G\agit (X,\Pp)$  and a $\G$-equivariant map $\pi : X\to\Omega$.
Assume that $\G \agit (X,\Pp)$ is ergodic and $\Pp[X_{\infty}]>0$. Then the considered generalized percolation has indistinguishable infinite clusters if and only if $R$ is  $\frac{\Pp}{\Pp[X_{\infty}]}$-ergodic.
\end{prop}

As a preliminary to the next subsection, we detail the proof of this theorem, which can be found in \cite{gl}.

\vspace{0.35 cm}

\begin{proof}
Assume that $R$ is ergodic.
Let $B$ be a $R_{cl}$-invariant Borel subset of $X$. Then, some $B'\in \{B,X\backslash B\}$ satisfies $\Pp[B'\cap X_\infty]=0$. Hence, $\Pp\left[\bigcup_{\gamma\in\Gamma}\gamma^{-1}\cdot(B'\cap X_\infty)\right]=0$, so that
$$
\Pp\left[\{x\in X:\forall y\in X_\infty\cap (\Gamma\cdot x),~ y\in X\backslash B'\}\right]=1.
$$
The first implication is thus a consequence of Proposition~\ref{one}.

The converse statement stems directly from the remark at the end of Subsection~\ref{indisclus} --- which makes crucial use of the ergodicity of $\Gamma\acts X$ --- and Proposition~\ref{two}.
\end{proof}

\subsection{Two lemmas on asymptotic invariance}

To translate properly the notion of strong ergodicity from orbit equivalence theory to percolation theory, we will need the following lemma. Since it holds with a high level of generality, and since the symbols $X$ and $R$ have a specific meaning in this section, we denote by $(Y,\mu)$ a standard probability space and by $R_Y$ a countable Borel equivalence relation on $Y$ that preserves the measure $\mu$.

\begin{lem}
\label{noinjectivity}
A sequence $(B_n)$ of Borel subsets of $Y$ is $\mu$-asymptotically $R_Y$-invariant if and only if for every Borel (not necessarily bijective) map $\phi : Y\to Y$ whose graph is included in $R_Y$, the $\mu$-measure of $\phi^{-1}(B_n)\triangle B_n$ converges to 0 as $n$ goes to infinity.
\end{lem}

\begin{rem}
This result is false if we replace $\phi^{-1}(B_n)$ with $\phi(B_n)$. Indeed, a Borel map whose graph is included in $R_Y$ may have a range of small measure. For instance, take the ``first-return in $[0,\epsilon[$ map'' for an action of $\Z$ on $\R/\Z \simeq [0,1[$ by irrational translation.
\end{rem}
\begin{proof}
One implication is tautological. To establish the other, assume that $(B_n)$ is asymptotically invariant and take $\phi$ a Borel map  from $Y$ to $Y$ whose graph is included in $R_Y$.
There exist 
\begin{itemize}
\item a partition $Y=\bigsqcup_{i\in\N} Y_i$ of $Y$ into countably many Borel subsets

\item and countably many $\varphi_i\in [R_Y]$
\end{itemize}
such that for all $i$, the maps $\phi$ and $\varphi_i$ coincide on $Y_i$. (This can be proved using Theorem~\ref{fm}.)
Let $\epsilon$ be a positive real number. Take $N$ such that $\mu\left(\bigsqcup_{i>N} Y_i\right)<\epsilon$.
For every $i$ and $n$, we have,
\begin{eqnarray*}
\phi^{-1}(B_n)\triangle B_n & \overset{\epsilon}{\simeq} & \bigsqcup_{i\leq N}  Y_i \cap (\phi^{-1}(B_n)\triangle B_n)\\
 & = & \bigsqcup_{i\leq N}  Y_i \cap (\varphi_{i}^{-1}(B_n)\triangle B_n)\\
 &\subset & \bigcup_{i\leq N} \varphi_{i}^{-1}(B_n)\triangle B_n,\\
\end{eqnarray*}
where $A\overset{\epsilon}{\simeq}B$ means that $\mu(A\triangle B)\leq \epsilon$.
Since  $\mu\left(\bigcup_{i\leq N} \varphi_{i}^{-1}(B_n)\triangle B_n\right)$ goes, by hypothesis, to 0 as $n$ goes to infinity, the lemma is established.
\end{proof}

We will also need the following lemma.

\begin{lem}
\label{restriction}
If $\Gamma\acts (Y,\mu)$ is a strongly ergodic action and if $Z$ is a Borel subset of $Y$ of positive measure, then $(Z,\frac{\mu}{\mu(Z)},(R_{\Gamma\acts Y})_{|Z})$ is strongly ergodic.
\end{lem}
\begin{rem}
If one replaces ``strongly ergodic'' with ``ergodic'' in the above statement, the proof is straightforward: one just needs to take $B$ an $R$-invariant set and apply ergodicity to $\Gamma\cdot B$. The proof gets a bit more technical in the strong case because one needs to take a \emph{suitable} $\Gamma$-saturation of $B$.
\end{rem}
\begin{proof}
Set $R:=(R_{\Gamma\acts Y})_{|Z}$. Let $(B_n)$ denote a $\frac{\mu}{\mu(Z)}$-asymptotically $R$-invariant sequence of Borel subsets of $Z$. It is enough to show that there exists a sequence $(B'_n)$ of $\mu$-asymptotically $\Gamma$-invariant subsets of $Y$ satisfying the following condition:
\begin{equation}
\label{eq}\tag{$\star$}
\mu(B_n\triangle(B'_n\cap Z))\underset{n\to\infty}{\longrightarrow}0.
\end{equation}
Indeed, by strong ergodicity of the action, the sequence $(\mu(B_n'))$ would then have no accumulation point other than 0 and 1, so that $\mu(B_n'\cap Z)$ would have no accumulation point other than 0 and $\mu(Z)$, which ends the proof together with condition (\ref{eq}).

For any finite subset $\Lambda$ of $\Gamma$, set
$$
B_{n,+}^{\Lambda} := \bigcap_{\gamma\in\Lambda} \gamma\cdot(B_n\cup(Y\backslash Z))~~~~\text{and}~~~~B_{n,-}^{\Lambda} := \bigcap_{\gamma\in\Lambda} \gamma\cdot((Z\backslash B_n)\cup(Y\backslash Z)).
$$
If $\Lambda$ is fixed and finite, the measure of $B_{n,+}^\Lambda\cup B_{n,-}^\Lambda$ converges to 1 as $n$ goes to infinity. 

\vspace{0.2 cm}

\begin{small}
Proceeding by contradiction, we assume that there exist $\eta$ and $\gamma$ in $\Lambda$ such that
$$
\limsup_n \mu(\{y\in Y : \eta\cdot y \in B_n\text{ and }\gamma\cdot y \in Z\backslash B_n\})>0.
$$
The measure $\mu$ being $\Gamma$-invariant, it follows that $$\limsup_n \mu(\{y\in Y : y \in B_n\text{ and }\gamma\eta^{-1}\cdot y \in Z\backslash B_n\})>0$$which contradicts the $\frac{\mu}{\mu(Z)}$-asymptotic $R$-invariance of $(B_n)$. More precisely, the mapping $\varphi : Z\to Z$ that sends $y$ to $\gamma\eta^{-1}\cdot y$ if the latter belongs to $Z$ and to $y$ otherwise contradicts Lemma~\ref{noinjectivity}.
\end{small}

\vspace{0.2 cm}

By a diagonal argument, one can find a sequence $(\Lambda_n)$ of finite subsets of $\Gamma$ such that, setting $\Lambda_n^{(2)}:=\{\gamma\eta:\gamma,\eta\in\Lambda_n\}$, the following two conditions hold:
\begin{itemize}
\item the sequence $(\Lambda_n)$ is non-decreasing and its union is $\Gamma$,

\item $\mu(B_{n,+}^{\Lambda_n^{(2)}}\cup B_{n,-}^{\Lambda_n^{(2)}})\underset{n\to\infty}{\longrightarrow}1$.
\end{itemize}
Set $B'_n:=B_{n,+}^{\Lambda_n}$.
For $n$ large enough, $\Lambda_n$ contains the identity element, so that 
$$
B_n\cap\left(B_{n,+}^{\Lambda_n}\cup B_{n,-}^{\Lambda_n}\right)=B_n\cap B_{n,+}^{\Lambda_n}=Z\cap B_{n,+}^{\Lambda_n}.
$$
It follows from this and the second condition that condition (\ref{eq}) is satisfied. To show that $(B'_n)$ is $\mu$-asymptotically $\Gamma$-invariant, take $\gamma\in\Gamma$. Taking $n$ large enough guarantees that $\gamma\in\Lambda_n$. The measure $\mu$ being $\Gamma$-invariant, we only need to show that $\mu(B'_n\backslash \gamma\cdot B'_n)$ tends to 0. To do so, it is enough to establish that the measure of $B'_n\backslash B_{n,+}^{\Lambda_n^{(2)}}$ tends to 0. Notice that
$$
B'_n\backslash B_{n,+}^{\Lambda_n^{(2)}}\subset Y\backslash\left(\left(B_{n,+}^{\Lambda_n^{(2)}}\cup B_{n,-}^{\Lambda_n}\right)\cap (\Lambda_n\cdot Z)\right).
$$

\vspace{0.2 cm}
\begin{small}Indeed, the sets $B_{n,+}^{\Lambda_n}\cap (\Lambda_n \cdot Z)$ and $ B_{n,-}^{\Lambda_n}\cap (\Lambda_n \cdot Z)$ are disjoint.\end{small}
\vspace{0.2 cm}

Since $Z$ has positive measure and $\Gamma\acts (Y,\mu)$ is ergodic, the measure of $\Lambda_n\cdot Z$ converges to 1. We conclude using the second condition.
\end{proof}

\subsection{Strong version}
\label{strong}

Consider $\Pp_p$ for $p \in (p_c(\mathcal{G}),p_u(\mathcal{G}))$.
By Theorems~\ref{corocf}, \ref{ls} and \ref{themm} and Proposition~\ref{ergind}, its cluster equivalence relation is strongly ergodic on the infinite locus.
One would like to deduce from this information a strong form of indistinguishability of $\Pp_p$. This idea is due to Damien Gaboriau. 

Another way to describe our goal is to say that we look for a proposition similar to Proposition~\ref{ergind} for strong notions. This is achieved in Theorem~\ref{ergindfort}.

Again, everything will be stated for a generalized percolation, with the same notation as previously.

\begin{defi}
We call \emph{re-anchoring}, or \emph{rerooting}, a  Borel map
$$
\begin{array}{lccl}
\alpha : & X\times V & \longrightarrow &  V \\
    & (x,v) & \longmapsto & u_{x,v}^\alpha \end{array}
$$
that is $\G$-equivariant under the diagonal action and such that $$\forall (x,v)\in X\times V,~ u_{x,v}^\alpha \overset{\pi(x)}{\longleftrightarrow} v.$$
\end{defi}

In words, a re-anchoring is a $\G$-equivariant way of changing of position within one's cluster.

\begin{exem}
If $\gamma\in \G$, setting 
$$u_{x,v}^{\alpha_\gamma} := \left\{
   	 \begin{array}{ll}
    	    \gamma\cdot v & \mbox{if } v\overset{\pi(x)}{\fleche}\gamma\cdot v \\
       	    v & \mbox{otherwise}
	    \end{array}\right. $$ defines a re-anchoring.
\end{exem}

\begin{defi}
Let $(P_n)$ be a sequence of vertex properties. Let $\Pp$ be a percolation. We will say that $(P_n)$ is an \emph{asymptotic cluster property} (for $\Pp$) if, for any rerooting $\alpha$,
\begin{equation*}
\label{convv}
\forall v\in V,~\Pp\left[\left\{x\in X:P_n(x,v)\Longleftrightarrow P_n\left(x,u^\alpha_{x,v}\right)\right\}\right]\underset{n\to\infty}{\longrightarrow}1.
\end{equation*}
\end{defi}

\begin{rem}
For a given rerooting, the convergence above holds for all $v$ as soon as it holds for one, by $\G$-invariance and -equivariance.
\end{rem}

\begin{rem}
This definition of ``depending asymptotically only on one's cluster'' is quite natural if one looks for a translation of strong ergodicity, but it may not be the clearest definition from a probabilistic point of view.
For a probabilistically more natural definition, see Subsection~\ref{natural}.
\end{rem}

\begin{nota}
In what follows, $A\Subset B$ means that $A$ is a finite subset of $B$.
\end{nota}

\begin{defi}
We will say that $\Pp$ satisfies the \emph{Strong Indistinguishability Property} if, for every $\Pp$-asymptotic cluster property $(P_n)$ and every $F\Subset V$,
$$
\Pp[P_n^\pm(x,\proptoop(x)\cap F)]\underset{n\to\infty}{\longrightarrow}1.
$$
\end{defi}

\begin{rem}
Subsection~\ref{natural} makes the definition of asymptotic cluster property look like the conclusion of strong indistinguishability.
\end{rem}

\begin{lem}
\label{bij}
The map $(B_n)\mapsto (P_{B_n})$ is a bijection from the set of the $\Pp$-asymptotically $R_{cl}$-invariant sequences of Borel subsets of $X$ onto the set of $\Pp$-asymptotic cluster properties. Its inverse map is $(P_n)\mapsto (B_{P_n})$.
\end{lem}

\begin{proof}
First, let $(B_n)$ be a $\Pp$-asymptotically $R_{cl}$-invariant sequence of Borel subsets of $X$ and set $P_n:=B_{P_n}$. We show that $(P_n)$ is a $\Pp$-asymptotic cluster property.

 Let $\alpha$ be a rerooting.
Since $(x,v)\mapsto (x,u_{x,v}^\alpha)$ is $\G$-equivariant, it induces a map $\overline{\alpha}:\G\backslash (X\times V)\to \G\backslash (X\times V)$. Set $$\phi := \psi\circ\overline{\alpha}\circ\psi^{-1},$$ where $\psi$ is the bijection introduced in Subsection~\ref{dico}. More explicitly, we have $\phi : x\mapsto \gamma_x^{-1}\cdot x$, where $\gamma_x$ is defined by $$u^\alpha_{x,\rho}=\gamma_x\cdot \rho.$$
The graph of this Borel map is a subset of $R$. By Lemma~\ref{noinjectivity}, the probability of $B_n\triangle\phi^{-1}(B_n)$ goes to 0 as $n$ goes to infinity. 
As a consequence, $(P_n)$ is an asymptotic cluster property.

\vspace{0.15 cm}

Now, let $(P_n)$ be a $\Pp$-asymptotic cluster property and set $B_n:=B_{P_n}$. We show that $(B_n)$ is $\Pp$-asymptotically $R_{cl}$-invariant.

Let $\phi \in [R]$. Since $R_{cl}\subset R_{\G\acts X}$, one can define a Borel map
$$
\begin{array}{l|ccl}
 & X_\infty & \longrightarrow & \G \\
    & x & \longmapsto & \gamma_x \end{array}
$$
such that
$
\forall x\in X,~ \phi(x)=\gamma_x^{~-1}\cdot x.
$
Define $\alpha$ by $u_{x,\eta\cdot \rho}^\alpha:=\eta\cdot \gamma_{\eta^{-1}\cdot x}$. This is a rerooting.
We have
\begin{eqnarray*}
\phi^{-1}(B_n)&=&\left\{x\in X : P_n(\phi(x),\rho)\right\} \\
&=&\left\{x\in X : P_n(\gamma_x^{-1}\cdot x,\rho)\right\} \\
&=&\left\{x\in X : P_n(x,\gamma_x\cdot\rho)\right\}\text{~~~~~~~by }\G\text{-invariance of }P_n\\
&=&\left\{x\in X : P_n(x,u^{\alpha}_{x,\rho})\right\} \\
\end{eqnarray*}
Since $(P_n)$ is a $\Pp$-asymptotic cluster property, we deduce from this that the probability of $B_n\triangle \phi^{-1}(B_n)$ tends to 0.
Since this holds for every $\phi\in[R]$, the sequence $(B_n)$ is ${\Pp}$-asymptotically $R_{cl}$-invariant.
\end{proof}

\begin{rem}
In the previous proof, the use of Lemma~\ref{noinjectivity} allows us to obtain the asymptotic-cluster-property condition for all rerootings, while a ``literal translation'' would have given it only for the vertex-bijective ones --- the rerootings $(x,v)\mapsto u_{x,v}$ such that, for every $x$, the map $v\mapsto u_{x,v}$ is bijective. From the percolation point of view, vertex-bijective rerootings are absolutely non-natural objects: the use of such a lemma was unavoidable.
\end{rem}
From Lemma~\ref{lemmaF} and Lemma~\ref{bij}, one deduces the following statement.

\begin{prop}
\label{strongone}
A generalized percolation satisfies the Strong Indistinguishability Property if and only if for every $\Pp$-asymptotically $R_{cl}$-invariant sequence $(B_n)$ of Borel subsets of $X$, for every $\Lambda\Subset \Gamma$,
$$
\Pp\left[\{x\in X:\forall y,z\in X_\infty \cap (\Lambda\cdot x), ~y\in B_n\Longleftrightarrow z\in B_n\}\right]\underset{n\to\infty}{\longrightarrow}1.
$$
\end{prop}

\begin{prop}
\label{strongtwo}
Consider a generalized percolation such that  $\Pp[X_{\infty}]>0$. 
The following assertions are equivalent:
\begin{enumerate}
\item the relation $R$ is $\frac{\Pp}{\Pp[X_\infty]}$-strongly ergodic,

\item for every asymptotic cluster property $(P_n)$, there exists $(\epsilon_n)\in\{-,+\}^\N$ such that
$$
\forall F \Subset V,~\Pp\left[P_n^{\epsilon_n}(x,\proptoop(x)\cap F)\right]\underset{n\to\infty}{\longrightarrow}1,
$$

\item for every asymptotic cluster property $(P_n)$, there exists $(\epsilon_n)\in\{-,+\}^\N$ such that
$$
\Pp\left[P_n^{\epsilon_n}(x,\rho)|\rho\overset{\pi(x)}{\fleche}\infty\right]\underset{n\to\infty}{\longrightarrow}1.
$$
\end{enumerate}
\end{prop}

\vspace{0.2 cm}

\begin{proof}
Assume that $R$ is strongly ergodic. Let $(P_n)$ be an asymptotic cluster property. Set $B_n:=B_{P_n}$. By strong ergodicity, there exists $(\epsilon_n)\in\{-,+\}^\N$ such that $\Pp[B^{-\epsilon_n}\cap X_\infty]$ tends to 0. (We denote by $B^+$ the set $B$ and $B^-$ its complement.) Hence, for any $\Lambda\Subset \Gamma$, $\Pp\left[\bigcup_{\gamma\in \Lambda} \gamma\cdot(B^{-\epsilon_n}\cap X_\infty)\right]$ tends to 0. This establishes the second statement: specifying the previous sentence for a particular $\Lambda$ solves the case $F=\Lambda^{-1}\cdot\rho$.

Taking $F=\{\rho\}$ gives $(\text{ii}) \implies (\text{iii})$ and $(\text{iii})\implies (\text{i})$ is straightforward.

\end{proof}

\begin{thm}
\label{ergindfort}
Consider a generalized percolation such that $\G \agit (X,\Pp)$ is strongly ergodic and $\Pp[X_{\infty}]>0$. It satisfies the Strong Indistinguishability Property if and only if $R$ is $\frac{\Pp}{\Pp[X_{\infty}]}$-strongly ergodic.
\end{thm}

\vspace{0.3 cm}

\begin{proof}
If $R$ is strongly ergodic, Proposition~\ref{strongtwo} implies that strong indistinguishability holds.
Conversely, assume strong indistinguishability to hold. Let $(B_n)$ be a $\frac{\Pp}{\Pp[X_\infty]}$-asymptotically $R$-invariant sequence of Borel subsets of $X_\infty$.
Strong indistinguishability implies that for every $\gamma$,
$$
\Pp\left[\{x\in X_\infty:\gamma\cdot x \in X_\infty \implies (x\in B_n\Longleftrightarrow \gamma\cdot x\in B_n)\}\right]\underset{n\to\infty}{\longrightarrow}\Pp[X_\infty].
$$
This means that $(B_n)$ is $\frac{\Pp}{\Pp[X_\infty]}$-asymptotically $(R_{\Gamma\agit X})_{|X_\infty}$-invariant. 
By Lemma~\ref{restriction}, the strong ergodicity of $R_{\Gamma\agit X}$ entails that the only possible accumulation points of $(\Pp[B_n\cap X_\infty])$ are $0$ and $\Pp[X_\infty]$. This ends the proof.

\end{proof}

\vspace{0.3 cm}

From this theorem and the few lines at the beginning of the current subsection, we can derive the following corollary --- even for $p=p_u(\mathcal G)$ if the assumption of the corollary is satisfied for this parameter.

\begin{coro}
\label{corollaire}
As soon as $\Pp_p$ produces infinitely many infinite clusters, it satisfies the Strong Indistinguishability Property.
\end{coro}

\subsection{Classic and strong indistinguishability do not coincide}

\label{donotcoincide}

Obviously,  strong indistinguishability implies the classic one: take $P_n=P$ for all $n$. In this subsection, we study a particular percolation, and prove that it satisfies the Indistinguishability Property but not the strong one.

To define this percolation, take $\Gamma$ to be the free group $\langle a, b\rangle$. Endow it with the generating system $\{a,b\}$. We will use the concrete definition of Cayley graphs and take the vertex set of $\mathcal{G}$ to be $\Gamma$. 
Set
$$X:=\{a,b\}^\Gamma~~~~~~~~~~\text{and}~~~~~~~~~~\Pp:=\left(\frac{1}{2}\delta_a+\frac{1}{2}\delta_b\right)^{\otimes \Gamma}.$$
The equivariant map $\pi$ is defined as follows: for each $\gamma$, among the two edges $\{\gamma,\gamma a\}$ and $\{\gamma, \gamma b\}$, open the edge $\{\gamma,\gamma x_\gamma\}$ and close  the other one. The analogous model for $\Z^2$ instead of $\langle a, b\rangle$ has been extensively studied, see e.g. \cite{finr} and references therein.

\begin{center}
\includegraphics[width=9.5 cm]{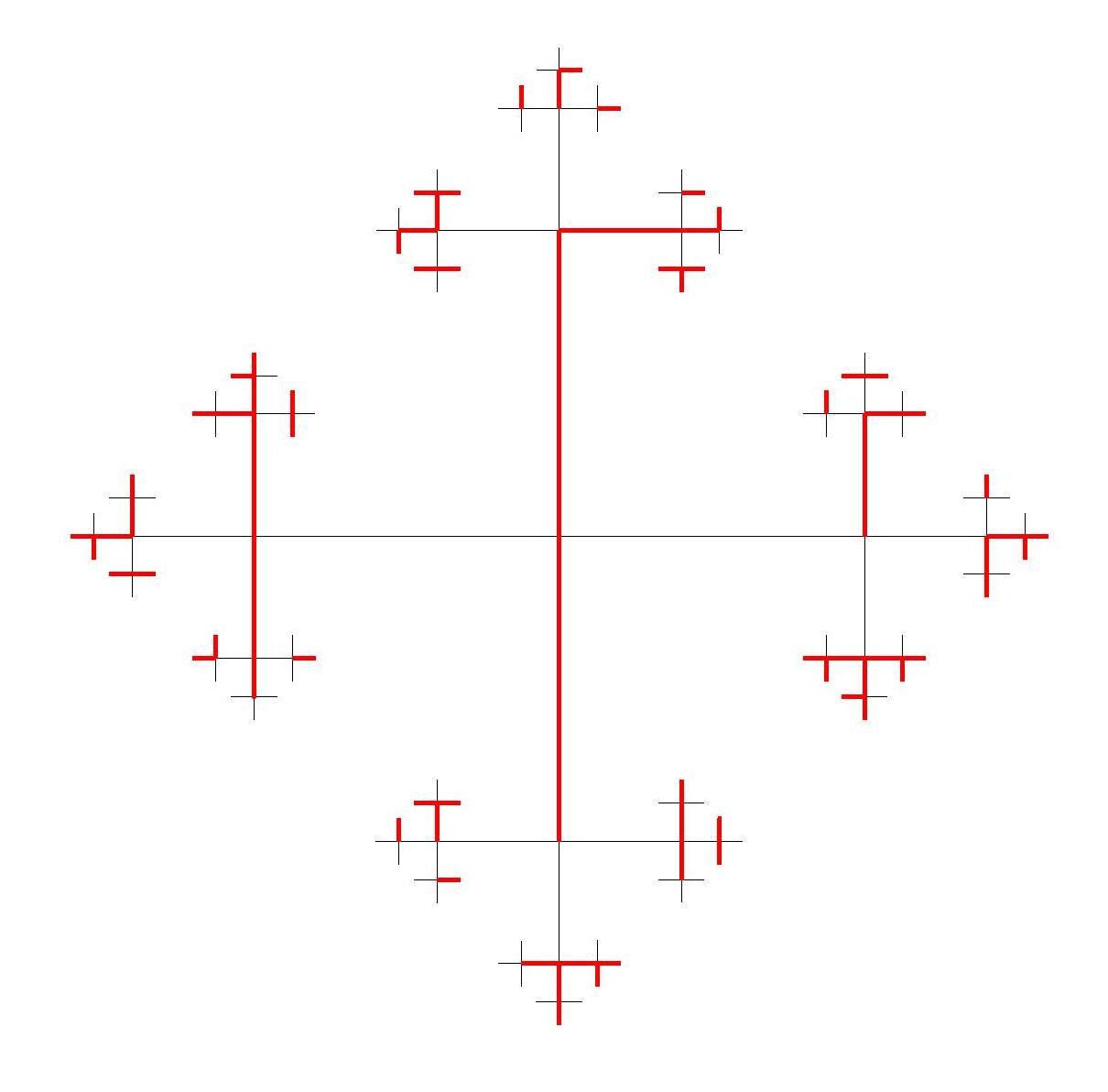}
\end{center}

\begin{thm}
The considered percolation satisfies the Indistinguishability Property but not the Strong Indistinguishability Property.
\end{thm}
\begin{proof}
In this proof, we will use the \emph{height} function defined as the unique morphism
$$
\begin{array}{lccl}
h : & \Gamma & \longrightarrow & \Z \\
    & a & \longmapsto & 1\\
    & b & \longmapsto & 1.\end{array}
$$

First, let us prove that strong indistinguishability does not hold. The $x$\emph{-directed path} launched at $\gamma$ is defined by $\tilde\gamma_0:=\gamma$ and $\tilde\gamma_{k+1}:=\tilde\gamma_k x_{\tilde\gamma_k}$. The elements $x_{\tilde\gamma_k}$ are called the \emph{steps} of the directed path.
Set $P_n(x,\gamma)$ to be ``there are more $a$'s than $b$'s in the first $2n+1$ steps of the $x$-directed path launched at $\gamma$''.
Let $d$ denote the graph distance on $\mathcal{G}$. Let $\gamma$ and $\eta$ denote two elements of $\Gamma$. Assume that there exists $x$ such that $\gamma$ and $\eta$ are $\pi(x)$-connected. Then, along the geodesic path from $\gamma$ to $\eta$, the height increases, reaches a unique maximum, and then decreases. Let $\tau$ be the vertex where this maximum is attained. If $\gamma$ and $\eta$ are $\pi(x)$-connected, the $x$-directed paths launched at $\gamma$ and $\eta$ coincide with the one launched at $\tau$, up to forgetting the first $d(\gamma,\tau)$ steps of the first path and the first $d(\eta,\tau)$ ones of the second. Thus, the probability of the event $$\gamma \overset{\pi(x)}{\fleche} \eta ~~~~~~\text{and}~~~~~~P_n(x,\gamma)\not= P_n(x,\gamma)$$ is less than the probability that a simple random walk on $\Z$ that takes $n-d(\gamma,\eta)$ steps ends up in $[-d(\eta,\gamma),d(\eta,\gamma)]$. This is known to go to zero as $n$ goes to infinity, as $n^{-1/2}$. Therefore, by Proposition~\ref{propopo}, $(P_n)$ is an asymptotic cluster property. But $P_n(x,a)$ and $P_n(x,b)$ are independent of probability $1/2$. Since the considered percolation produces only infinite clusters, it cannot satisfy the Strong Indistinguishability Property.

Now, let us establish the Indistinguishability Property.
Let us define the \emph{contour} exploration of the cluster of the origin $\rho=1$. Intuitively, we explore the cluster of the origin (and some vertices of its boundary) using a depth-first search algorithm, with the following conventions:
\begin{itemize}
\item vertices of negative height are ignored,

\item when a vertex $\gamma$ has its two sons $\gamma a^{-1}$ and $\gamma b^{-1}$ in its cluster, $\gamma a^{-1}$ is explored first --- in figures, $\gamma a^{-1}$ will be represented to the left of $\gamma b^{-1}$.
\end{itemize} Formally, the exploration is defined as follows. If $m$ is an integer, define $\vec{E}_{x,m}$ to be
$$
\left\{(\gamma,\gamma s^{-1}):\gamma\in\Gamma,~ s\in\{a,b\},~ h(\gamma)> m\right\}\cup\left\{(\gamma,\gamma x_\gamma):\gamma\in\Gamma,~h(\gamma)\geq m\right\}.
$$
Given a configuration $x\in\{a,b\}^\Gamma$, we define a bijection $\mathsf{next}_{x,m}$ from $\vec{E}_{x,m}$ to itself. If $(\gamma,\gamma')\in \vec{E}_{x,m}$, then $\mathsf{next}_{x,m}(\gamma,\gamma')$ is set to be $(\gamma',\gamma'')$, where $\gamma''$ is

\vspace{0.6 cm}

\begin{tabular}{ll}
$\gamma'b^{-1}$& if $\gamma'=\gamma a$,\\
$\gamma' a^{-1}$& if $\gamma=\gamma' x_{\gamma'}$ and $h(\gamma') > m$,\\
$\gamma$& if $\gamma\not=\gamma' x_{\gamma'}$ and $h(\gamma)=h(\gamma')+1$,\\
$\gamma'x_{\gamma'}$& otherwise.
\end{tabular}

\vspace{0.6 cm}

\begin{center}
\includegraphics[height=5.4cm]{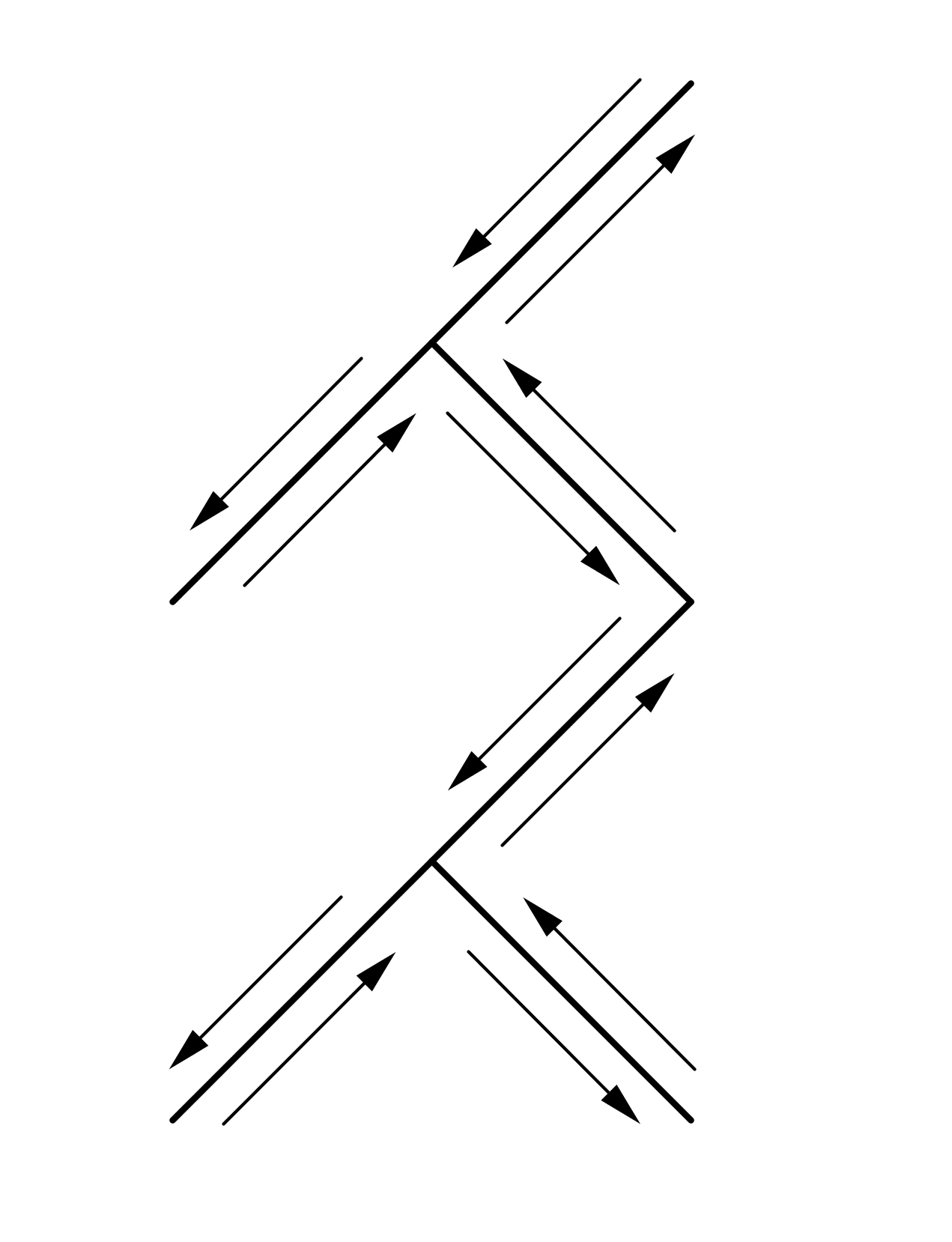}
\end{center}

The \emph{exploration} --- or \emph{exploration in positive time} --- is defined by
\begin{itemize}
\item $\vec{e}_0=(\gamma_0,\gamma_1)=(1,x_1)$,

\item $\forall k> 0,~\vec{e}_k=(\gamma_k,\gamma_{k+1})=\mathsf{next}_{x,0}(\vec{e}_{k-1})$.
\end{itemize}
Since $\mathsf{next}_x$ is a bijection, one can also define the \emph{exploration in negative time}:
\begin{itemize}
\item $\vec{e}_0=(\gamma_0,\gamma_1)=(1,x_1)$,

\item $\forall k\leq 0,~\vec{e}_k=(\gamma_k,\gamma_{k+1})=\mathsf{next}_{x,0}(\vec{e}_{k-1})$.
\end{itemize}
Whenever there is no explicit mention of negative times, ``exploration'' will always be understood as ``exploration in positive time''.
Define  $$k(x):=\min\left\{k>0:h(\gamma_k)=0\text{ and }\gamma_k\overset{\pi(x)}{\fleche}1\right\}.$$ Notice that it is almost surely well-defined.

\vspace{0.3 cm}

\begin{small}
Indeed, for each positive height $n$, there exists a unique couple $(\gamma_{n,x},\gamma'_{n,x})$ satisfying the following conditions:
\begin{itemize}
\item the $x$-directed path launched at 1 contains $\gamma_{n,x}'$ but not $\gamma_{n,x}$,
\item $\gamma_{n,x}^{-1}\gamma_{n,x}'\in\{a,b\}$
\item and $h(\gamma_{n,x})=n$.
\end{itemize}
Denote by $T_{n,x}$ the connected component of $\gamma_{n,x}$ in the graph defined by $\pi(x)$ but where the edges $\gamma_{n,x}a$ and $\gamma_{n,x}b$ have been removed. It is rooted at $\gamma_{n,x}$. The following facts hold:
\begin{itemize}
\item considered as rooted graphs up to isomorphism, the $T_{n,x}$'s are i.i.d.\ critical Galton-Watson trees,
\item each $T_{n,x}$ has probability 1/4 of being explored\footnote{Of course, the generations of negative height are not explored.} by the contour exploration (it has probability 1/2 of belonging to the cluster of 1 and, conditionned on this, it has probability 1/2 of being explored in positive time rather than in negative time)
\item and the events and random variables mentionned in the two facts above are independent.
\end{itemize}
Since the depth of a critical Galton-Watson tree is non-integrable, by the independent form of the Borel-Cantelli Lemma, it almost surely occurs that one of them is explored and reaches height 0.
\end{small}

\vspace{0.3 cm}

Thus, the Borel mapping $x\mapsto \gamma_{k(x)}^{-1}\cdot x$ coincides on a full-measure set with a Borel bijection $T:X\to X$.

\vspace{0.3 cm}

\begin{small}
Indeed,  $k'(x):=\min\{k<0:h(\gamma_k)=0\text{ and }\gamma_k\overset{\pi(x)}{\fleche}1\}$ is almost surely well-defined, so that the mapping $S:x\mapsto \gamma_{k'(x)}^{-1}\cdot x$ is almost surely well-defined. For almost every $x$, $T(S(x))=S(T(x))$.
\end{small}

\vspace{0.3 cm}

For almost every $x$, the points $T(x)$ and $x$ are in the same $\Gamma$-orbit. By Theorem~\ref{preserv}, the Borel bijective map $T$ preserves the measure $\Pp$.
By Proposition~\ref{ergind}, it is enough to show that $T$ is ergodic. (Indeed, for almost every $x$, the point $T(x)$ and $x$ are in the same $R_{cl}$-class.)

Let $B$ denote a Borel subset of $X$ and assume that $B=T(B)$. We need to show that $\Pp[B]\in\{0,1\}$. Let $\epsilon >0$. Let $C$ be an event such that
\begin{itemize}
\item $\Pp[B\triangle C]<\epsilon$,

\item $C$ is $\sigma(x_{|\mathcal{B}})$-measurable for some ball $\mathcal{B}$ centered at 1.
\end{itemize}
Denote by $R$ the radius of the ball $\mathcal{B}$ and by $\textbf{C}$ the subset of $\{a,b\}^\mathcal{B}$ such that
$$
C=\textbf{C}\times\prod_{\gamma\not\in\mathcal{B}} \{a,b\}.
$$
Set $X_n:=T^n(x)_{|\mathcal{B}}$. We will show that $(X_n)_{n\geq 0}$ is an irreducible aperiodic time-homo\-ge\-neous Markov chain. Assuming this, we conclude the proof. Since $\Pp$ is $T$-invariant, it would result from our assumption that
$$
\Pp[X_0\in \textbf{C}\text{ and }X_n\in\textbf{C}] \underset{n\to\infty}{\longrightarrow}\Pp[X_0\in \textbf{C}]^2.
$$
Using the notation $A\overset{\epsilon}{\simeq} A'$ as a shortcut for $\Pp[A\triangle A']\leq\epsilon$, we have
$$
B=B\cap T^n(B)\overset{2\epsilon}{\simeq} C\cap T^n(C).
$$
Letting $n$ go to infinity, we get $|\Pp[B]-\Pp[C]^2|\leq2\epsilon$. Since $|\Pp[C]-\Pp[B]|<\epsilon$, we have $|\Pp[B]-\Pp[B]^2|<4\epsilon$. Letting $\epsilon$ go to zero, one gets $\Pp[B]=\Pp[B]^2$ and concludes.

Now, let us prove that $(X_n)$ is an irreducible aperiodic time-homo\-ge\-neous Markov chain.
Since $(X_n)$ is defined by iteration and restriction of the measure-preserving transformation $T$, if it is a Markov chain, it is necessarily time-homo\-ge\-neous. Let us establish the Markov Property.

To define $(X_0,\dots,X_n)$, one needs to explore a certain set of vertices denoted by $\mathsf{Explo}_n(x)$.
\begin{itemize}
\item Conditionally on $(X_0,\dots,X_n)$, the state of the vertices in $\Gamma\backslash\mathsf{Explo}_n(x)$ is i.i.d. $\frac{1}{2}\delta_a+\frac{1}{2}\delta_b$.

\item Define $\hat\gamma_0$ to be the point of height $R+1$ in the $x$-directed path launched at $1$. Then, define an auxiliary exploration: it explores the vertices of the cluster of the origin as previously until it reaches $\hat\gamma_0 x_{\hat\gamma_0}$ and then executes the exploration defined by $\mathsf{next}_{x,R+1}$. Notice that, after $\hat\gamma_0$, the vertices explored by the auxiliary exploration are exactly the ones of height at least $R+1$ that are explored by the usual exploration; besides, they are explored in the same order.  Denote by $(\hat\gamma_k)$ the sequence of the vertices of height exactly $R+1$ that are visited by any of our two explorations, in the order in which they are discovered. Set $\mathcal{P}$ to be the set of the elements of $\Gamma$ whose expression as a reduced word starts with $a^{-1}$ or $b^{-1}$. Conditionally on the data of the whole auxiliary exploration, the sequence $$\left((\hat\gamma_k^{-1}\cdot x)_{|\mathcal{P}}\right)_{k\geq 1}$$ is i.i.d., the common law of its elements being $\left(\frac{1}{2}\delta_a+\frac{1}{2}\delta_b\right)^{\otimes\mathcal{P}}$.
\end{itemize}

The exploration never visits a site of $\hat\gamma_k\cdot\mathcal{P}$ after one of $\hat\gamma_\ell\cdot\mathcal{P}$ for $\ell > k$. Thus, to establish the Markov Property, it is enough to show that, within some $\hat\gamma_k\cdot\mathcal{P}$, the vertices that we explore between the $n^\text{th}$ and $(n+1)^\text{th}$ steps of the construction (in order to define $X_{n+1}$) and that have already been explored have their state written in $X_n$. More formally, it is enough to show that if we set
\begin{itemize}
\item $k_-:=\min\{k\leq 0 : \gamma_k=\hat \gamma_0\}$,
\item $k_+:=\max\{k\geq 0 : \gamma_k=\hat \gamma_0\}$,
\item $\mathcal{L}:=\{\gamma_k : k_- \leq k \leq 0\}\backslash\{\hat \gamma_0\}$,
\item $\mathcal{L}':=\{\eta:\exists \gamma \in \mathcal{L},~ h(\gamma)=0\text{ and }d(\gamma,\eta)\leq R\}$,
\item $\mathcal{R}:=\{\gamma_k : 0 \leq k \leq k_+\}\backslash\{\hat \gamma_0\}$,
\item $\mathcal{R}':=\{\eta:\exists \gamma \in \mathcal{R},~ h(\gamma)=0\text{ and }d(\gamma,\eta)\leq R\}$
\end{itemize}
then $(\mathcal{L}\cup \mathcal{L}')\cap(\mathcal{R}\cup \mathcal{R}')$ is always included in $\mathcal{B}$. Since $\mathcal{L}'\cap \mathcal{R}'$ consists in the $1+R$ first vertices visited by the $x$-directed path launched at $1$, it is a subset of $\mathcal{B}$. 
To establish $\mathcal{L}\cap \mathcal{R}'\subset\mathcal{B}$, take $\gamma$ in $\mathcal{L}$ and $\eta$ at height 0 such that $\eta\in\mathcal{R}$ and $d(\gamma,\eta)\leq R$. It results from the respective definitions of $\mathcal{L}$ and $\mathcal{R}$ that the geodesic path connecting $\gamma$ to the tripod $(1,\hat\gamma_0,\eta)$ intersects it at a point $\kappa$ which belongs to the geodesic $(1,\hat\gamma_0)$. Since $1$ and $\eta$ have the same height, $d(\kappa,1)\leq d(\kappa,\eta)$. Thus $d(\gamma,1)\leq d(\gamma,\eta)$ and $\mathcal{L}\cap \mathcal{R}'\subset\mathcal{B}$.

\vspace{0.6 cm}

\begin{center}
\includegraphics[height=4cm]{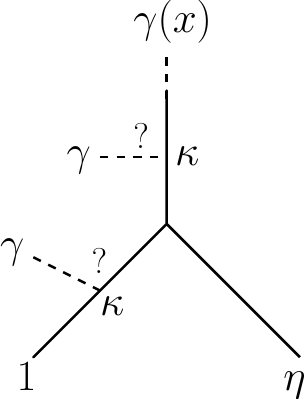}
\end{center}

\vspace{0.6 cm}

The inclusion $\mathcal{L}'\cap \mathcal{R}\subset \mathcal{B}$ follows by symmetry. To have the Markov Property, it remains to show that $\mathcal{L}'\cap \mathcal{R}'\subset \mathcal{B}$. This results from the fact that if $\gamma\in\mathcal{L}$ and $\eta\in\mathcal{R}$ both have height 0, then every point $\kappa$ of the tree spanned by $\{\hat\gamma_0,\gamma,1,\eta\}$ satisfies $$d(\kappa,1)\leq\max(d(\kappa,\gamma),d(\kappa,\eta)).$$

\vfill

\begin{center}
\includegraphics{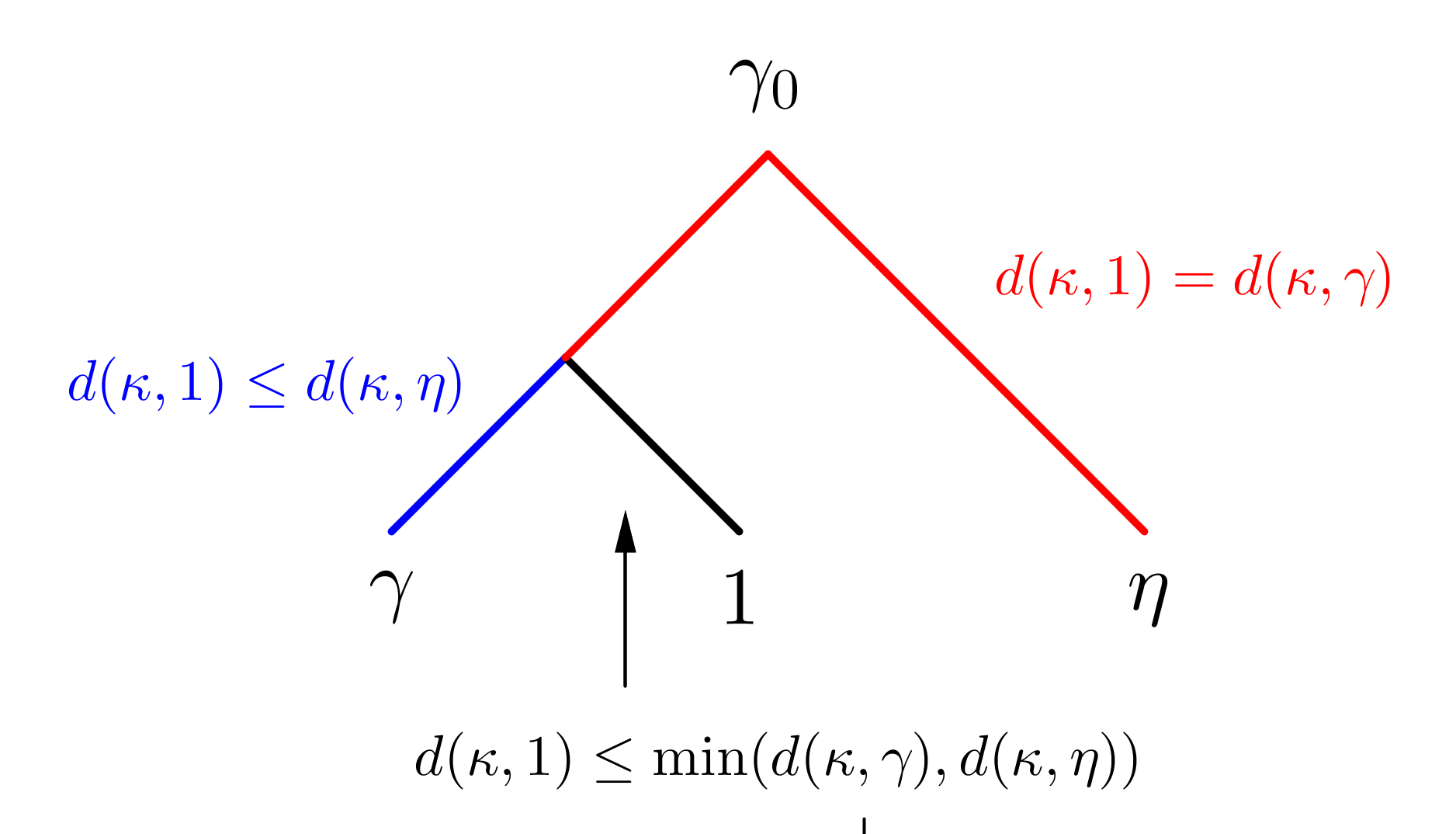}
\end{center}

\vfill

\begin{center}
\includegraphics{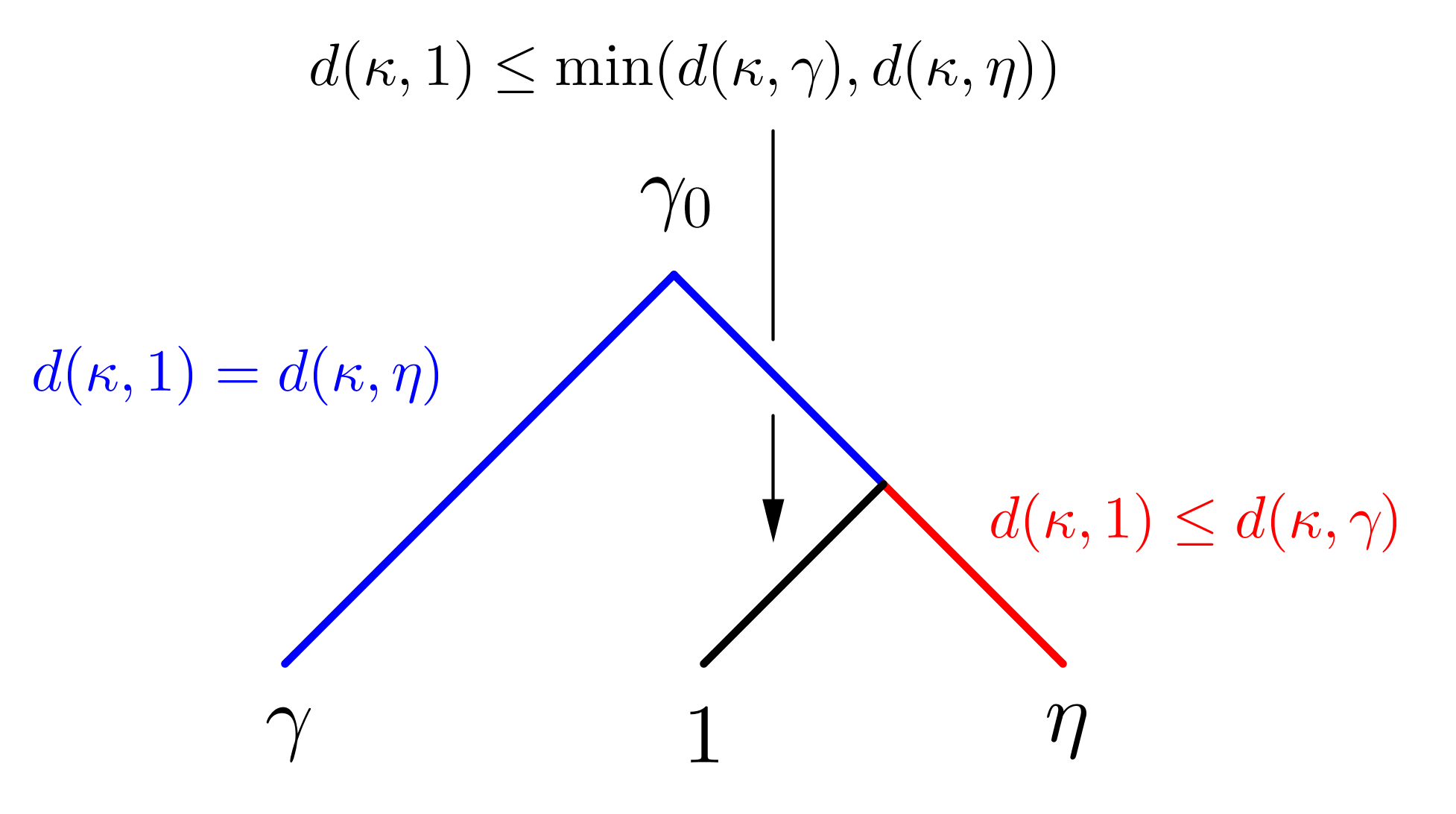}
\end{center}

Now, let us establish the irreducibility of the considered Markov chain. Let $\chi$ and $\xi$ be two elements of $\{a,b\}^\mathcal{B}$. The knowledge of the restriction of $x$ to $\mathcal{B}$ suffices to determine the point at height $R+1$ in the $x$-directed path launched at $1$. Denote it by $\gamma(x_{|\mathcal{B}})$. Imposing on $x$ the following conditions (compatible since they involve disjoint areas)
\begin{itemize}
\item $x_{|\mathcal{B}}=\chi$,
\item $x_{\gamma(\chi)}=a$,
\item $x_{\gamma(\chi)ab^{-1}}=b$,
\item $\left(\gamma(\xi)ba^{-1}\gamma(\chi)^{-1}\cdot x\right)_{|\mathcal{B}}=\xi$.
\end{itemize}
we have $X_0=\chi$ and $\exists k>0, X_k=\xi$. Thus, the intersection of these two events has positive probability and $(X_n)$ is irreducible.

\vspace{0.6 cm}

\begin{center}
\includegraphics[width=5cm]{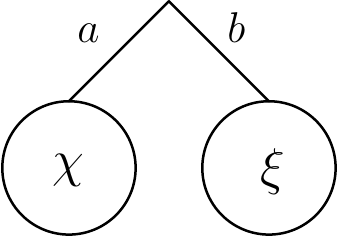}
\end{center}

\vspace{0.6 cm}

To establish the aperiodicity of the Markov chain $(X_n)$, apply the previous argument for $\chi=\xi=(a)_{\gamma\in\mathcal{B}}$ with the additional condition $x_{a^{n+1}b^{-1}}=a$, which gives $\Pp[X_0=X_1=(a)_{\gamma\in\mathcal{B}}]>0.$
\end{proof}
\begin{rem}
The previous proof does not only prove that the infinite clusters are indistinguishable, but also that the ``height-levels of infinite clusters'' are indistinguishable, which is a stronger statement.
\end{rem}

\subsection{Complements on asymptotic cluster properties}
\label{natural}

This subsection provides equivalent definitions of asymptotic cluster properties. We stick to the usual notation for generalized percolations.

\begin{nota}
If $x\in X$, denote by $\clusters(x)$ the set of the clusters of $\pi(x)$.
\end{nota}

\begin{prop}\label{propopo} Let $(P_n)$ be a sequence of properties.
The following assertions are equivalent
\begin{enumerate}
\item $(P_n)$ is a $\Pp$-asymptotic cluster property,

\item $\forall F \Subset V, \Pp\left[\forall C \in \clusters(x), P_n^\pm(x,C\cap F)\right]\underset{n\to\infty}{\longrightarrow}1$,

\item $\exists u\in V, \forall v\in V, \Pp[P_n^\pm(x,\{u,v\})|u\overset{\pi(x)}{\fleche}v]\underset{n\to\infty}{\longrightarrow}1$,

\item $\forall u\in V, \forall v\in V, \Pp[P_n^\pm(x,\{u,v\})|u\overset{\pi(x)}{\fleche}v]\underset{n\to\infty}{\longrightarrow}1$.
\end{enumerate}
\end{prop}

\begin{rem}
Above, we set $P[A|B]:=1$ when $\Pp[B]=0$.
\end{rem}
\begin{proof}
The assertions (iii) and (iv) are equivalent by $\G$-invariance.

Rewriting (ii) as follows
$$
\forall F \Subset V, \Pp\left[\forall (u,v)\in F^2, \left(u\overset{\pi(x)}{\fleche}v\right) \implies P_n^\pm(x,\{u,v\})\right]\underset{n\to\infty}{\longrightarrow} 1
$$
clarifies its equivalence\footnote{Remember that $\Pp[Q_n|Q]\underset{n\to\infty}{\longrightarrow}1$ is equivalent to $\Pp[Q\implies Q_n]\underset{n\to\infty}{\longrightarrow}1$.} with (iv): one way, take $F := \{u,v\}$; the other way, write $F$ as the \emph{finite} union of the pairs it contains.

Now assume (i) and establish (iii). We will do so for $u=\rho$.
Let $v=\gamma\cdot\rho$ be a vertex. 
Applying (i) to the $\alpha_\gamma$ introduced at the beginning of Subsection~\ref{strong}, one gets
$$
\Pp\left[\left\{x\in X : P_n(x,\rho)= P_n\left(x,u^{\alpha_\gamma}_{x,\rho}\right)\right\}\right]\underset{n\to\infty}{\longrightarrow}1.
$$
Hence, if $A:=\{x\in X: \rho\overset{\pi(x)}{\fleche}\gamma \cdot \rho\}$,
$$
\Pp\left[\left\{x\in A : P_n(x,\rho)= P_n\left(x,u^{\alpha_\gamma}_{x,\rho}\right)\right\}\right]\underset{n\to\infty}{\longrightarrow}\Pp[A].
$$
But, on $A$, ``$P_n(x,\rho)= P_n\left(x,u^{\alpha_\gamma}_{x,\rho}\right)$'' means that ``$P_n(x,\rho)= P_n(x,v)$'', so that (iii) is established.
\vspace{0.15 cm}

It is now enough to show that (ii) implies (i). Assume (ii). Let $\alpha$ be a rerooting. Set $w(x) := u^\alpha_{x,\rho}$ and take $\epsilon > 0$. Let $F\Subset V$ be such that $\Pp[w\not\in F]<\epsilon$.
We have
$$
(w \in F\text{ and }\forall C\in\clusters(x),P_n^\pm(x,F\cap C))\implies P_n^\pm\left(x,\left\{\rho, w\right\}\right).
$$
\begin{small}(Apply the second hypothesis to the common cluster of $\rho$ and $w$.)\end{small}

\vspace{0.15 cm}

The condition on the left hand side being satisfied with probability asymptotically larger than $1-2\epsilon$ (by (ii) and choice of $F$), $$\liminf_{n} \Pp\left[ P_n^\pm\left(x,\left\{\rho, w\right\}\right) \right]\geq 1-2\epsilon.$$ Since this holds for any value of $\epsilon$, the proof is over.
\end{proof}

\vfill

S\'ebastien Martineau

UMPA, ENS de Lyon

46 all\'ee d'Italie

69 364 Lyon Cedex 07

\textsc{France}

sebastien.martineau@ens-lyon.fr
\end{document}